\documentclass{amsart} 
\usepackage{graphicx}
\usepackage{eucal}  
\usepackage{graphicx}  
\usepackage{enumerate}  
\usepackage{amsmath,amssymb}  
\usepackage[table]{xcolor}
\usepackage{amscd}  
\usepackage{amsfonts}    
\usepackage{mathrsfs}  
\usepackage{xtab}
\usepackage{hyperref}  
\renewcommand\ge{\geqslant}    
\renewcommand\le{\leqslant}    
\renewcommand\bar{\overline}  
  
\newcommand\N{\mathbb N}  
  
\newcommand\R{\mathbb R}  
\newcommand\Z{\mathbb Z}  
\newcommand\C{\mathbb C}  
\newcommand\pn[1]{\mathbb{P}^{#1}}  
\renewcommand\span{\operatorname{Span}}

\newcommand\relint{\operatorname{relint}}
\newcommand\cone{\operatorname{Cone}}
\newcommand\conv{\operatorname{Conv}}  
\newcommand\Cay{\operatorname{Cayley}}  
\newcommand\F{\mathbb{F}}   
\newcommand\Bl{\operatorname{Bl}}
\newcommand\Star{\operatorname{Star}}
\renewcommand\O{\mathcal{O}}
\newcommand\A{\mathcal{A}}

\newtheorem{theorem}{Theorem}
\newtheorem{example}{Example}
\newtheorem{remark}{Remark}
\newtheorem{conjecture}{Conjecture}
\newtheorem{lem}[theorem]{Lemma}
\newtheorem{defi}[theorem]{Definition}
\newtheorem{prop}[theorem]{Proposition}
\newtheorem{coro}[theorem]{Corollary}

\begin{document}

\title{A classification of smooth convex 3-polytopes with at most 16 lattice points}


\author{Anders Lundman}\footnote{KTH, Department of Matematics. The author is supported by the V.R. grant NT:2010-5563.} 




\maketitle

\begin{abstract}
We provide a complete classification up to isomorphism of all smooth convex lattice 3-polytopes with at most 16 lattice points. There exist in total 103 different polytopes meeting these criteria. Of these, 99 are strict Cayley polytopes and the remaining 4 are obtained as inverse stellar subdivisions of such polytopes. We derive a classification, up to isomorphism, of all smooth embeddings of toric threefolds in $\pn N$ where $N\le 15$. Again we have in total 103 such embeddings. Of these, 99 are projective bundles embedded in $\pn N$ and the remaining 4 are blow-ups of such toric threefolds.
\end{abstract}
This is a preprint. The final version of the article is published at The Journal of Algebraic Combinatorics Online First, 2012, DOI: 10.1007/s10801-012-0363-3 and is available at \url{http://www.springerlink.com/content/m756421425168170/}

\section{Introduction}  
There exists a fascinating correspondence between convex lattice polytopes and embeddings of toric varieties via complete linear series. In particular two embeddings of toric varieties are isomorphic if and only if the corresponding polytopes are isomorphic, i.e. if they differ by a lattice preserving affine isomorphism. 
Let $M\cong \Z^d$, recall that a $d$-dimensional convex lattice polytopes $P\subset M\otimes \R$ is called smooth if there are exactly $d$ edges through every vertex of $P$ and the edge-directions form a lattice basis for $M$. A $d$-dimensional toric variety embedded in $\pn {k}$ is smooth if and only if the corresponding $d$-dimensional convex lattice polytope is smooth (see \cite{Cox2010tv} for details). 

It has recently been proven in \cite{BHH10}    that for any $  d,k\in \Z^+$    there are, up to isomorphism, only finitely many smooth convex lattice $d$-dimensional polytopes $  P\subset \R^d$     such that $  |P\cap \Z^d|\le k$. By the correspondence mentioned above this implies that for a fixed choice of $  d,k\in \Z^+$    there are, up to isomorphism, only finitely many embeddings of smooth toric varieties of dimension $d$  into $\pn {k-1}$. An alternative proof for this theorem has also been given in \cite{lundman}. From an elaboration of the proof given in \cite{BHH10}   a complete classification of all smooth convex lattice $  d$-dimensional polytopes $  P\subset \R^d$   such that $  |P\cap \Z^d|\le 12$    has been constructed by Lorenz in \cite{Lorenz}    for $  d=2$     and  $  d=3$. The classification given by Lorenz relies on extensive calculations using the program Polymake. 

In this paper we utilize Lorenz' classification of all smooth 2-dimensional convex lattice polytopes to obtain a classification of all smooth 3-dimensional convex lattice polytopes $  P$    such that $  |P\cap \Z^3|\le 16$   as well as the corresponding toric embeddings. We prove that:
\begin{theorem}\label{CBconc}
Up to isomorphism there exists exactly 103 smooth 3-dimensional convex lattice polytopes $P\subset \R^3$ such that $|P\cap \Z^3|\le 16$. Equivalently there are, up to isomorphism, 103 $\pn k$-embeddings of smooth 3-dimensional toric varieties such that $k\le 15$.
\end{theorem}

The smooth 3-dimensional convex lattice polytopes with at most 12 lattice points appearing in our classification coincide exactly with the 3-dimensional polytopes in the classification given in \cite{Lorenz}. Our classification is  obtained by analyzing the geometrical constraints imposed by the  hypothesis. A key step in our approach is Lemma \ref{V14} in which we prove that any smooth 3-dimensional convex lattice polytope $P$ such that $|P\cap \Z^3|\le 16$ has at most 8 facets. This makes it possible to use the classification of triangulations of the 2-sphere given in \cite[p.59]{Oda} to get the number of edges in the facets of any 3-dimensional convex lattice polytope meeting our restrictions. 
 
The polytopes and embeddings appearing in our classification fall naturally into four categories; see Section \ref{notation}. We will prove Theorem \ref{CBconc} in two steps. First we show in Proposition \ref{Odacoro} that any smooth 3-dimensional convex lattice polytope $P$ Êsuch that $|P\cap \Z^3|\le 16$ Êhas to lie in one of the four categories. Sections \ref{sectCay} and \ref{sectListBl} are then devoted to classifying all polytopes meeting our restrictions in each of the four categories. A complete list of polytopes and embeddings can be found in the appendix. 

This paper is based on the authors master thesis at the Department of Mathematics at KTH in Stockholm.
\section{Notation and Background}\label{notation}
Let $  P$   be a $d$-dimensional convex lattice polytope in $  \R^n$   and let $  \Sigma$   be the inner-normal fan of $  P$. The polytope $P$ defines an embedding of a $d$-dimensional toric variety $  X_\Sigma$ in $\pn {k-1}$ where $k=|P\cap \Z^d|$. Such embeddings will be called complete embeddings as they are defined by the complete linear system of the associated ample line bundle. For more details we refer to \cite{Cox2010tv} and \cite{Ewald}. In this paper we will call a convex lattice polytope of dimension $d$ simply a $d$-\emph{polytope}. Moreover a strongly convex rational polyhedral cone will be called simply a \emph{cone} and a complete polyhedral fan is called simply a \emph{fan}. 

\begin{defi}
Let $N\cong \Z^d$ be a lattice and $\sigma$ be a $d$-dimensional cone in $N\otimes_\R \R$. We call $\sigma$ \emph{unimodular} if there exist $d$ lattice vectors $v_1,\dots, v_d \in N$ such that $\sigma$ is the positive linear span of $v_1,\dots,v_d$ in $N\otimes_\R \R$ and $v_1,\dots,v_d$ form a lattice basis for $N$. A fan $\Sigma$ is called \emph{unimodular} if all cones in $\Sigma$ are unimodular.
\end{defi}

Smoothness of a toric variety can be defined in a strict algebraic geometry setting \cite{Cox2010tv}. In fact the following statements are equivalent.
\begin{prop}[{\cite[\S 2.1]{Fulton1993itt}}]\label{tripplesmooth}  
Let $  P$   be a full dimensional polytope with inner-normal fan $  \Sigma$   and let $  X_\Sigma$   be the toric variety defined by $  \Sigma$, then the following is equivalent
\begin{enumerate}  [i)]
\item{$  P$   is smooth}  
\item{$  \Sigma$   is unimodular}  
\item{$  X_\Sigma$   is smooth.}  
\end{enumerate}  
\end{prop}  

\subsection{Stellar subdivisions and blow-ups}
For more details on blow-ups we refer to \cite[\S VI 7.]{Ewald}. Recall that the \emph{relative interior} of a cone $\sigma \subset \R^n$ is the set of points $x\in\sigma$ such that there exist some ball $B\subset \sigma$ containing $x$. For a given cone $\sigma\subset \R^n$ we denote the relative interior of $\sigma$ by $\relint(\sigma)$. Given two cones $\sigma, \tau$ we denote by $\cone(\sigma\cup \tau)$ the cone spanned by the union of the spanning vectors in $\sigma$ and $\tau$. 

\begin{defi}
Let $\sigma$ be a cone and $\Sigma$ a fan. Assume that $\span(\sigma)\cap \span(\sigma')=\{0\}$ Êfor every $\sigma' \in \Sigma$ and that $\relint(\cone(\sigma\cup \sigma'))\cap \relint(\cone(\sigma \cup \sigma'')=\emptyset$ for all $\sigma', \sigma''\in \Sigma$ such that $\sigma'\ne \sigma''$. The \emph{join} of $\sigma$ and $\Sigma$ is defined as
\[
\sigma \cdot \Sigma:=\{\cone(\sigma \cup \tau): \forall \tauÊ\in \Sigma\}.
\]
\end{defi}
Let $\Sigma$ be a fan then the \emph{star} Êof a cone $\sigma \in \Sigma$ Êis the set $\Star_\Sigma(\sigma):=\{\tau \in \Sigma:\sigmaÊ\text{ is a face of }\tau\}$. 
The \emph{closed star} of $\sigma\in \Sigma$ is the set $\bar{\Star_\Sigma(\sigma)}=\{\tau \in \Sigma: \tau \text{ is a face of } \tau'\in\Star_\Sigma(\sigma)\}$.
\begin{defi}\label{stellsub}
Let $\Sigma$ be a fan, $\sigma\in \Sigma$ be a cone, $p\in \operatorname{relint}(\sigma)$ be a point and $\rho=\R_{\ge 0}p$ be the ray spanned by $p$. We call the set 
\[
s(\Sigma;\rho):= \left( \Sigma \setminus \operatorname{Star}_\Sigma(\sigma)\right) \cup \rho \cdot  \left(\overline{\operatorname{Star}_\Sigma(\sigma)}\setminus\operatorname{Star}_\Sigma(\sigma)\right)
\]
the \emph{stellar subdivision} of $\Sigma$ in direction $p$. The fan $\Sigma$ is called the \emph{inverse stellar subdivision} of $s(\Sigma;\rho)$.
\end{defi}
By $\F_r$ we denote the \emph{Hirzebruch surface} $\pn { }(\O_{\pn 1} \oplus \O_{\pn 1}(r))$. Recall that the defining fan of $\F_r$ is $\Sigma_r:=\{\{0\},\cone((1,0)), \cone((0,1)),\cone((0,-1)),$\newline $\cone((-1,r)),\cone((1,0),(0,1)),\cone((1,0),(0,-1)),\cone((0,-1),(-1,r)),$\newline$\cone((-1,r),(0,1))\}$ illustrated below.

\begin{figure}[ht!]
\includegraphics[scale=0.6]{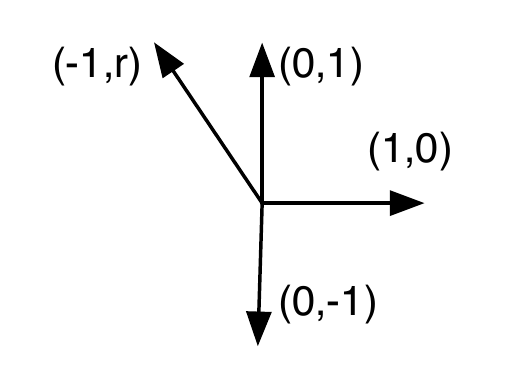}
\caption{The fan of $\F_r$}
\end{figure}

\begin{example}
Let $\Sigma$ be the fan of $\pn 2$. The fan $\Sigma_1$ is the stellar subdivision of $\Sigma$ in direction $p=(0,1)$.
\end{example}

The stellar subdivision of a unimodular fan in direction $p$ is called \emph{unimodular} if $s(\Sigma,\rho)$ in Definition \ref{stellsub} is unimodular. Remember our convention that fans are always complete. 
\begin{remark} For a  fan $\Sigma\subset \R^n$ and a unimodular stellar subdivision $s(\Sigma,\rho)\subset \R^n$ the identity map $\R^n \to \R^n$  is a map of fans $ s(\Sigma,\rho)\to \Sigma$ and therefore induce a toric morphism $X_{s(\Sigma,\rho)}\mapsto X_\Sigma$. Note that $X_{s(\Sigma,\rho)}$ is a blow-up of $X_\Sigma$.

\end{remark}

All unimodular stellar subdivisions of a unimodular fan $\Sigma$ are characterized by the following theorem.
\begin{theorem}\label{unich}
Let $\Sigma$ be a unimodular fan and assume that $\sigma=\cone(x_1,\dots,x_r)\in \Sigma$ is a cone where $x_1,\dots,x_r$ are linearly independent lattice vectors that generate $\sigma\cap \Z^n$. Let $\rho=\R_{\ge 0}p$ where $p$ generates $\rho \cap \Z^n$. Then $s(\rho;\sigma)$ is an unimodular stellar subdivision if and only if 
\[
p=x_1+\dots+x_r
\]
\end{theorem}
\begin{proof} See \cite[p.179]{Ewald}\end{proof}
Let $\Sigma$ be a unimodular fan and let $\sigma\in \Sigma$. As a consequence of Theorem \ref{unich} we write $s(\sigma)=s(\Sigma,\rho)$ for a unimodular stellar subdivision $s(\Sigma,\rho)$ of $\Sigma$.  We refer to the blow-up associated to $s(\sigma)$ as the blow-up of $X_\Sigma$ at $X_{\check{\sigma}}$. 

Let $\Sigma$ be the inner-normal fan of a smooth polytope $P$ and let $L_P$ be the associated ample line bundle whose global sections define the corresponding embedding, see \cite[Chapter 6]{Cox2010tv} for more details. Consider a unimodular stellar subdivision $\Sigma'$ given by $s(\sigma)$ where $\sigma\in \Sigma$. Let $F$ be the face of $P$ associated to $\sigma$ and $\pi:X_{\Sigma'} \rightarrow X_\Sigma$ be the induced blow-up map with exceptional divisor $E$. When the line bundle $\pi^*L_P-kE$ is ample for $k\ge 1$, it defines a polytope $P'$ obtained by cutting off the face $F$ at level $k$. We will denote the polytope $P'$ by $\Bl_F^k(P)$. Consider for example $(X_P,L_P)=(\pn 3,\O_{\pn 3}(3))$ Êand choose a fixed point $p$ corresponding to the vertex $v$ in $P$. Then the polytopes associated to the blow-ups of $(X_P,L_P)$ Êat the fixed point $p$ corresponding to a vertex $v$ of $P$ are as illustrated below.

\begin{figure}[h!]
\centering
\includegraphics[scale=0.5]{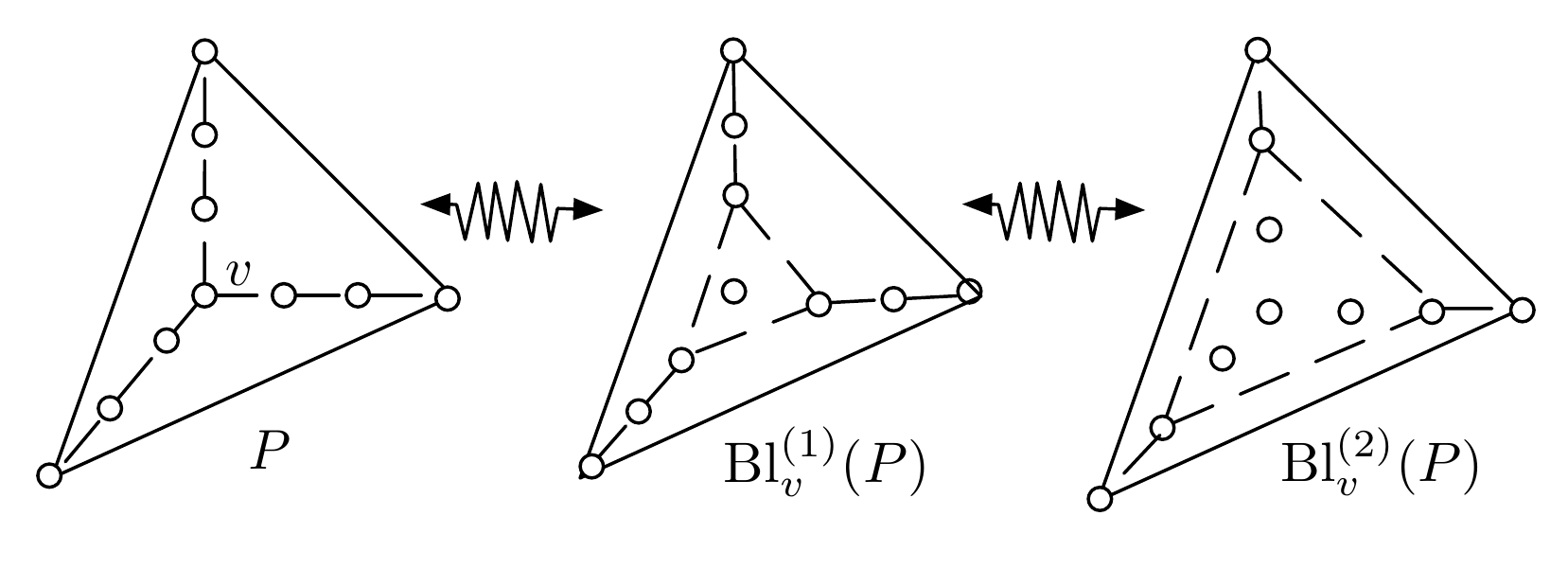}
\caption{The blow-ups of $(\pn 3,\O_{\pn 3}(3))$ at a fixed point $p$ corresponding to a vertex $v$ of $P$.}
\end{figure}

Note that here $(X_P,L_P)\leftrightsquigarrow P$, $(Bl_p(\pn 3), \pi^*(\O_{\pn 3}(3))-E) \leftrightsquigarrow  \Bl_v^{(1)}(P)$ and $(Bl_p(\pn 3), \pi^*(\O_{\pn 3}(3))-2E) \leftrightsquigarrow  \Bl_v^{(2)}(P)$. To shorten notation the family of polarized toric varieties obtained by consecutive blow-ups of $(X_P,L_P)$ at $n$ torus invariant subvarieties will be denoted  $\Bl_n(X_P)$. The corresponding family of polytopes obtained via consecutive blow-ups of $P$ at $n$ faces will be denoted by $\Bl_n(P)$.  
\begin{defi}
Let $P$ be a smooth $d$-polytope. We call $P$ \emph{minimal} if it cannot be obtained as blow-up along a face of an other smooth $d$-polytope.
\end{defi}
Note that a polytope is minimal if and only if the corresponding embedded toric variety is minimal in the sense of equivariant blow-ups. 

\subsection{Toric fibrations and Cayley polytopes} 

Remember that two polytopes are \emph{strictly isomorphic} if they have the same inner-normal fan. 

\begin{defi}  
Let $  P_0,\dots ,P_k \subset \R^n$      be strictly isomorphic polytopes with inner-normal fan $  \Sigma$. Let $  \{e_1,\dots,e_k\}  $   be a basis for $  \Z^k$   and let $  e_0=(0,\dots, 0)\in \Z^k$. Then a polytope $P$ is called a \emph{$s$-th order strict Cayley polytope}   associated to $  P_0, \dots, P_k$      if it is isomorphic to
\[
\Cay^s_\Sigma(P_0,\dots,P_k):=\conv((P_0,se_0),\dots, (P_k,se_k))\subset \R^{n+k}  
\]
where $  s\in \Z^+$   and $  (P_i,se_i):=\{(p,se_i) : p\in P_i\}  $.
\end{defi}  
\begin{example}\label{Cay2d}
A 2-polytope $P\subset \R^2$ is strictly Cayley if and only if $P$ is isomorphic to either $s\Delta_2$ or $\Cay^s_\Sigma(P_0,P_1)$  where $P_0$ and $P_1$ are line segments and $\Sigma$ is the fan associated to $\pn 1$.
\end{example}
\begin{example}  \label{Cay3d}  
A 3-polytope $  P\subset \R^3$   is strictly Cayley if and only if it is of one of the following three types
\begin{enumerate}  [i)]
\item{$  P \cong s\Delta_3$   where $  s\in \Z^+$    and $  \Delta_3$   is the standard simplex.}  
\item{$  P\cong \Cay^s_\Sigma(P_0,P_1,P_2)$   where $  P_0, P_1$      and $  P_2$      are line segments, and $  \Sigma$   is the fan associated to $  \pn 1$.}  
\item{$  P\cong \Cay^s_\Sigma(P_0,P_1)$   where $  P_0$      and $  P_1$   are strictly isomorphic 2-polytopes with inner-normal fan $  \Sigma$.}  
\end{enumerate}  
\end{example}  

The following three lemmas follows directly from the definition of a smooth polytope.
\begin{lem}  \label{type3restriction2}  
Let $  P\cong \Cay^s_\Sigma(P_0,P_1)$    be a d-polytope,  where $  P_0$      and $  P_1$      are strictly isomorphic and smooth (d-1)-polytopes. Then $  P$      is smooth if and only if there are exactly $s+1$ lattice points on every edge between $(P_0,0)$ and $(P_1,s)$.\end{lem}  

\begin{lem}  \label{sdivides}  
Let $  P\subset \R^3$ be an $s$-th order strict Cayley 3-polytope of the type $\Cay^s_\Sigma(P_0,P_1,P_2)$   where $P_0=[0,i]$, $P_1=[0,j]$   and $  P_2=[0,k]$. Then $  P$   is smooth if and only if s divides $  j-i$, $  k-i$   and $  k-j$. In particular every first order strict Cayley 3-polytope of this type is smooth.
\end{lem}  

\begin{lem}  \label{type2restriction}  
Let $  P\cong\Cay^s_\Sigma(P_0,P_1,P_2)$,   where $  P_0\cong[0,i]$, $  P_1\cong[0,j]$   and $  P_2\cong[0,k]$      are line segments, be a 3-polytope such that $  |P\cap \Z^3|\le 16$. Then $  s\le 2$   and $i+j+k\le 13$.
Moreover up to isomorphism we may assume that $  i\ge j \ge k$   and $  k\le 4$. 
\end{lem}  

For the following two definitions remember that with a \emph{fan} we mean a complete polyhedral fan. 
\begin{defi}
Let $\Sigma$ and $\Sigma'$ be fans and assume that the join $\sigma \cdot \Sigma'$ exist for all $\sigma \in \Sigma$ and that 
\[
\relint(\sigma \cdot \sigma')\cap \relint(\tau \cdot \tau')=\emptyset
\]
for all $\sigma, \tau \in \Sigma$ and $\sigma', \tau' \in \Sigma'$ such that $\sigma \ne \tau$ and $\sigma' \ne \tau'$. Then we call the set 
\[
\Sigma \cdot \Sigma' :=\{\sigma \cdot \sigma' : \sigma \in \Sigma, \sigma' \in \Sigma'\}
\]
\textit{the join of $\Sigma$ and $\Sigma'$}.
\end{defi}

\begin{defi}  
Let $  \Sigma\subseteq \R^n$ be a unimodular fan. Assume that $\Sigma$ is the join of a unimodular fan $  \Sigma'$ which covers a lower-dimensional linear subspace $U$ of $\R^n$  and a unimodular fan $  \Sigma''$. The projection $\pi:\R^n \to \R^n/U$ induce a map of fans $\pi:\Sigma'' \to \pi(\Sigma'')=:\Sigma_\pi$. We call $  X_\Sigma$     a $  X_{\Sigma'}  $-\textit{fiber bundle over}   $  X_{\Sigma_\pi}  $    under the surjection $  \bar{\pi}  :X_\Sigma \to X_{\Sigma_\pi}  $   induced by the projection $  \pi:\Sigma\to \Sigma_\pi$. 
\end{defi}  

An easy corollary of the results presented in \cite{Dickenstein2009}    and \cite{CasagrandeCDiRocco}    is the following proposition, which is most useful for us.
\begin{prop}  \label{CAYisPk}  
Let $  P$   be the smooth polytope associated to an embedding of a toric variety $  X_P$. Then $  P\cong \Cay^s_{\Sigma}(P_0,\dots, P_k)$   where the strictly isomorphic polytopes $  P_0, \dots , P_k$   have inner-normal fan $  \Sigma$   if and only if $  X_P$   is a $  \pn k$-fiber bundle over $  X_{\Sigma}$.
\end{prop}  

We are now ready to state Theorem \ref{CBconc} in full detail. 
\begin{theorem}  \label{CBconclusion}  
Up to isomorphism there exist 103 smooth 3-polytopes $P\subset \R^3$     such that $  |P\cap \Z^3|\le 16$. Equivalently there are, up to isomorphism, 103 complete embeddings of toric threefolds into $  \pn k$   such that $  k\le15$. All such pairs of 3-polytopes and embeddings may be arranged into the following four categories.
\begin{enumerate}  [i)]
\item{$  P\cong k \Delta_3$ where $k=1,2$ and $  X_P\cong \pn 3$    embedded in either $  \pn 3$       or $  \pn 9$. }  
\item{$  P\cong \Cay^s_{\Sigma}  (P_0,P_1,P_2)$,    where $  P_0$, $  P_1$       and $  P_2$      are line segments and  $  X_P$      is a $  \pn 2$-fiber bundle over $  \pn 1$    embedded in $  \pn N$,  $  5\le N \le 15$, $1\le s\le 2$.}  
\item{$  P\cong \Cay^s_\Sigma(P_0,P_1)$, where $  P_0$      and $  P_1$     are strictly isomorphic smooth 2-polytopes and $  X_P$     is a $  \pn 1$-fiber bundle either over $  \pn 2$, over $\Bl_2(\pn 2)$, over $  \Bl_3(\pn 2)$ or over the Hirzebruch surface $\F_r$,  with $  0\le r \le 4$,   embedded in $  \pn N$   where $  5\le N\le 15$, $1\le s\le 3$.}  
\item{$P$   is  the blow-up of a strict Cayley polytope at one, two or four vertices and does not lie in category \textit{i)-iii)}. The corresponding toric variety $X_P$ is either the blow-up of $\pn 3$ at four points embedded in $\pn {15}$Êor the blow-up of a $\pn 2$-fiber bundle over $  \pn 1$ at one or two points embedded in $\pn {13}, \pn {14}$   or $  \pn {15}  $.}  
\end{enumerate}
Up to isomorphism 99 of the 103 pairs are in category \textit{i)}, \textit{ii)}   and \textit{iii)}. In particular all complete embeddings of smooth toric varieties that are minimal in the sense of equivariant blow-ups lie in categories \textit{i)}, \textit{ii)}   and \textit{iii)}.
\end{theorem}
Note that since $\Cay^s_\Sigma(k\Delta_2,k\Delta_2)\cong \Cay^k(s\Delta,s\Delta,s\Delta)$ categories \textit{ii)} and \textit{iii)} of Theorem \ref{CBconclusion} are not mutually exclusive. Note moreover that Theorem \ref{CBconc} follows directly from Theorem \ref{CBconclusion}.



\section{Oda's classification}\label{Odaclass}  

Our approach to prove Theorem \ref{CBconclusion} is to first prove the following proposition.

\begin{prop}  \label{Odacoro}  
Let $  P\subset \R^3$   be a smooth 3-polytope such that $  |P\cap \Z^3|\le 16$. If $  P$   is minimal then $P$  is a strict Cayley 3-polytope.
\end{prop}  

Recall that a minimal smooth toric surface is isomorphic to either $\pn 2$ or to a Hirzebruch surface. Note that the toric surfaces $\pn 2$ and $\F_r$ correspond exactly to the 2-dimensional strict Cayley polytopes. Theorem \ref{CBconclusion} states that a smooth 3-polytope with at most 16 lattice points is either strictly Cayley or the blow-up of a smooth strict Cayley 3-polytope. Hence our results reveal that, with the added bound to the number of lattice points, the set of all smooth 3-polytopes have an underlying Cayley structure analogous to smooth 2-polytopes. \begin{defi}     
Let $  \Sigma\subset \R^3$    be a fan. The intersection $  \Sigma\cap S^2$      of $  \Sigma$    with the unit sphere $  S^2$       is called a \textit{spherical cell complex}. If for every cone $  \sigma \in \Sigma$,   \textit{the spherical cell}   $  \sigma\cap S^2$   is a triangle drawn on $  S^2$,   then we call $  \Sigma\cap S^2$   a \textit{triangulation}   of the unit sphere. 
\end{defi}  

Note that every full-dimensional cone in a unimodular fan  is a simplex cone. Therefore, since we only consider complete fans, the spherical cell complex  $  \Sigma\cap S^2$   is a triangulation of $  S^2$ for any unimodular fan $\Sigma$ in $\R^3$. 

\begin{defi}\label{deflabel}
Let $  T$      be a triangulation of $  S^2$      and let $  v_m$   be the number of vertices in $  T$   with degree $  m$. Then we associate to $  T$    the label
\[
\prod_{m>0}   m^{v_m}. 
\]
as a word in the alphabet $\Z^+$. 
Note that the number of vertices in $  T$   can be read off as the sum $  \sum_{m>0}   v_m$.
\end{defi}  

The main reason we are interested in triangulations of $  S^2$      is the following lemma which readily follows from definition \ref{deflabel} (for details see \cite[p.52]{Oda}).  
\begin{lem}  \label{TtoSLPH}  
Let $  \Sigma$   be the inner-normal fan of a simple 3-polytope $  P$      and assume that the triangulation $  T$   associated to $  \Sigma$   has the label $  m_1^{v_1}  \cdots m_k^{v_k}  $. Then $  P$ has $  v_i$   facets with $  m_i$   edges for all $  i\in \{1,\dots ,k\}  $. 
\end{lem}  

A complete classification of all combinatorial types of triangulations of the unit sphere containing at most 8 spherical cells is presented in \cite[p.192]{Oda}. The following theorem is a translation of Theorem 1.34 stated in \cite[p.59]{Oda} into the language used in this paper. 

\begin{theorem}[{\cite[p.59]{Oda}}]\label{ODAthmrv}  
Let $P$ be a smooth and minimal 3-polytope with at most 8 facets. If $  \Sigma$  is the inner-normal fan of $  P$, then the label of the triangulation $\Sigma \cap  S^2$  will be one of the following.
\[
3^4, 3^24^3, 4^6,4^55^2, 4^66^2,3^24^36^2,3^14^35^3,3^24^47^2\]
\[3^34^15^16^3,3^24^25^26^2,3^14^45^16^2,3^24^15^46^1,3^14^35^46^1,3^25^6,4^45^4
\]
with weights as in \cite{Oda}. Moreover for the first 5 triangulations the associated toric varieties $  X_\Sigma$   are $  \pn k$-fibrations as follows. 
\begin{enumerate}  [i)]
\item{If $  T=3^4$       then $  X_\Sigma\cong \pn 3$  }  
\item{If $  T=3^24^3$       then $  X_\Sigma$     is a $  \pn 1$-fiber bundle over $  \pn 2$   or a $  \pn 2$-fiber bundle over $  \pn 1$  }  
\item{If $  T=4^6$       then $  X_\Sigma$      is a $  \pn 1$-fiber   bundle over the Hirzebruch surface $  \F_a$     where $  a>1$     or $  a=0$.}  
\item{If $  T=4^55^2$    then $  X_\Sigma$       is a $  \pn 1$-fiber bundle over a smooth toric variety associated to a smooth pentagon.}  
\item{If $  T=4^66^2$      then $  X_\Sigma$     is a $  \pn 1$-fiber bundle over a smooth toric variety associated to a smooth hexagon.}  
\end{enumerate}  
\end{theorem}

From Proposition \ref{CAYisPk}   we get the following corollary of Theorem \ref{ODAthmrv}.
\begin{coro}  \label{Odacorollary}  
Let $  P$   be a smooth and minimal 3-polytope with at most 8 facets. If the triangulation associated to $  P$   is $  3^4, 3^24^3, 4^6, 4^55^2$   or $  4^66^2$   then $  P$   is strictly Cayley.
\end{coro}  

 Our next objective is to prove Proposition \ref{Odacoro}. We will do this in two steps. The first step is to prove that any smooth 3-polytope $  P$,   such that $  |P\cap \Z^3|\le 16$,   has at most 8 facets. The second step is to prove that if $  P$   is  minimal, then the triangulation associated to $  P$   will be $  3^4, 3^24^3, 4^6, 4^55^2$   or $  4^66^2$. Proposition \ref{Odacoro}  then follows directly from Corollary  \ref{Odacorollary}.
 
\begin{lem}  \label{EFrest}  
A simple 3-polytope  with $  V$      vertices has $  3V/2$   edges and $  2+V/2$      facets.
\end{lem}  

\begin{proof}See \cite[\S10.3]{Grunbaum}.\end{proof}

In \cite{Lorenz}   Lorenz provides a complete list of all smooth 2-polytopes $  P$   such that $  |P\cap \Z^2|\le 12$. We are particularly interested in the $n$-gons of that classification with $n\ge 5$. For conveniency of the reader we list these as a separate lemma below.

\begin{lem}[See \cite{Lorenz}]\label{Ngons}
Let $P$ be a smooth 2-polytope with at least 5 edges such that $|P\cap \Z^{2}|\le 12$ then $P$ is isomorphic to one of the following polytopes:

\begin{figure}[h!]\label{Ngonsfig}
\begin{tabular}  {|r|r|r|r|}  
\hline
\includegraphics[scale=0.25]{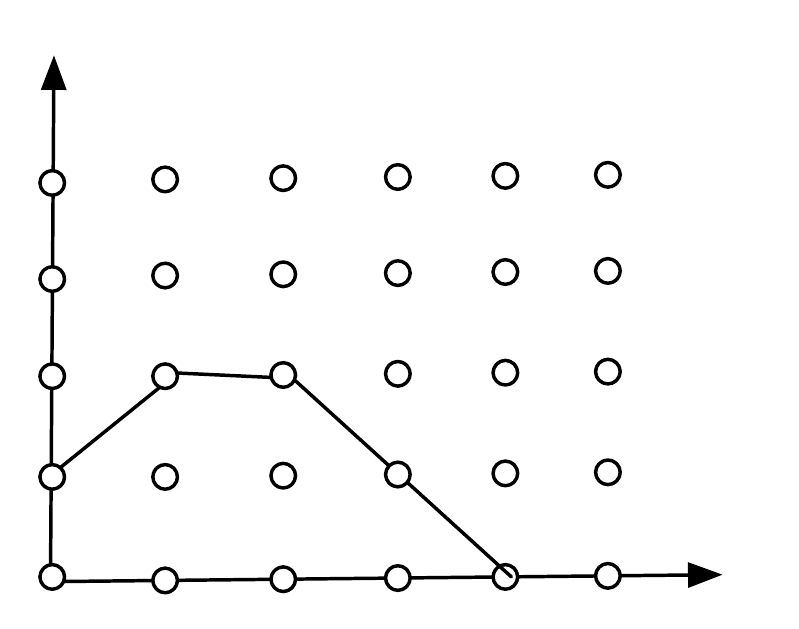} &\includegraphics[scale=0.25]{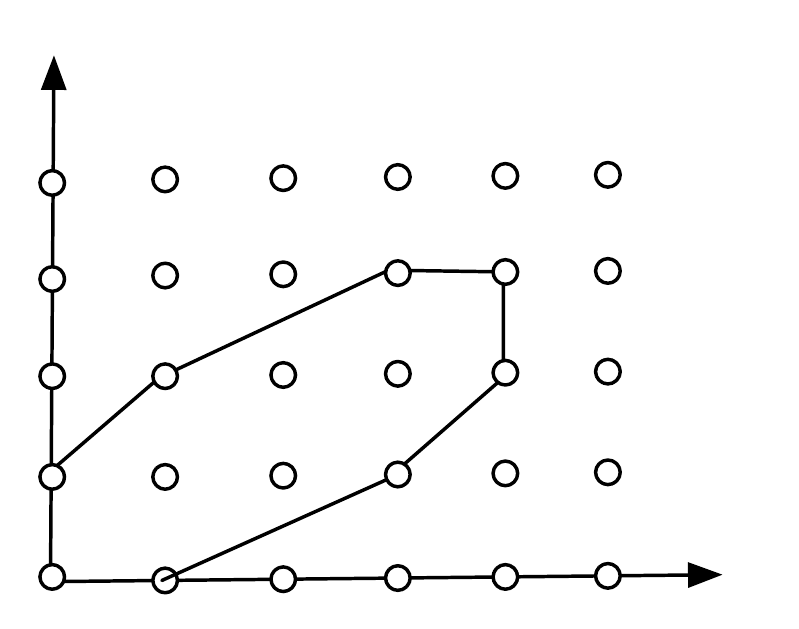}&\includegraphics[scale=0.25]{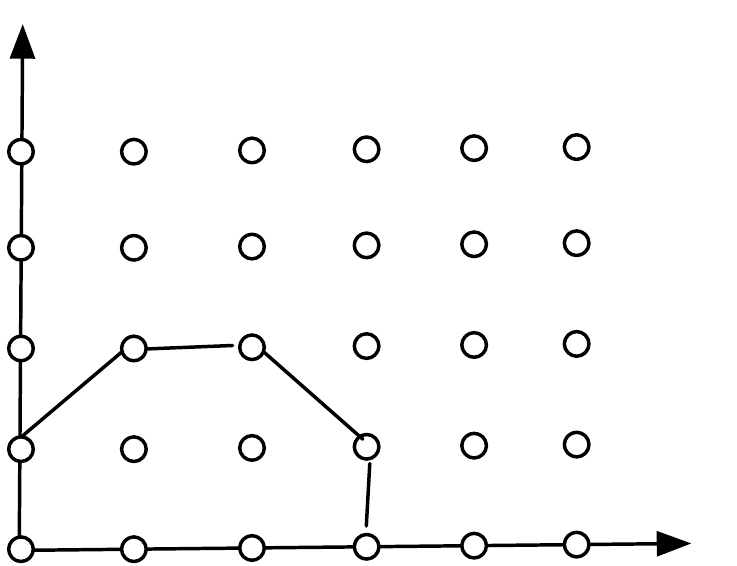}&\includegraphics[scale=0.25]{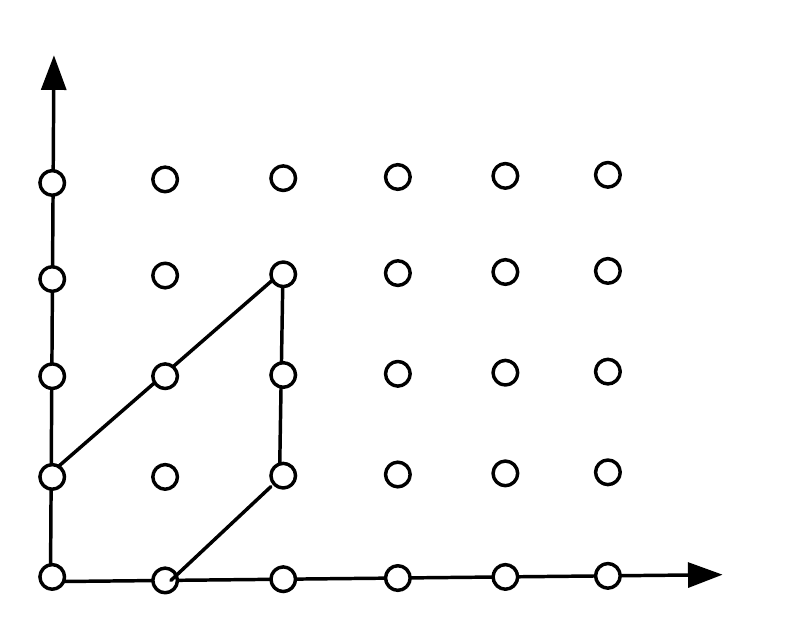}\\ \hline\includegraphics[scale=0.25]{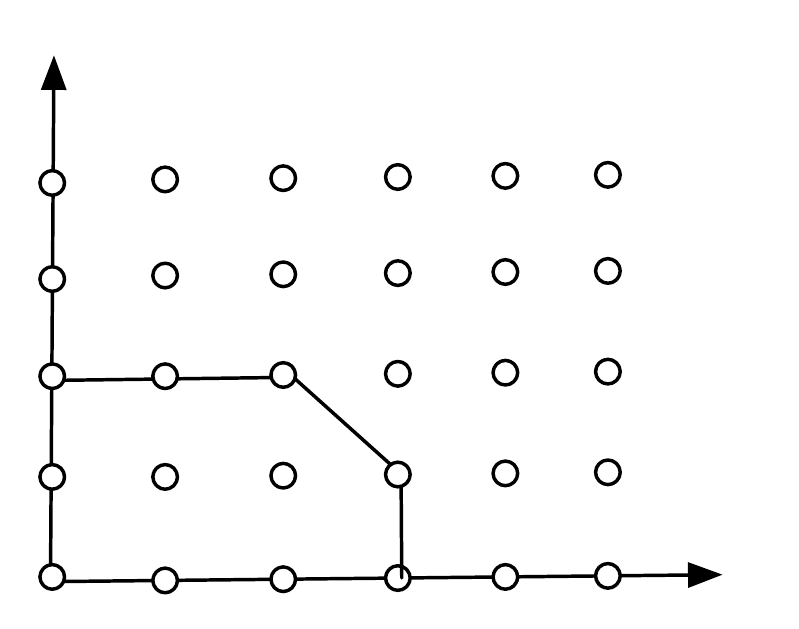}&\includegraphics[scale=0.25]{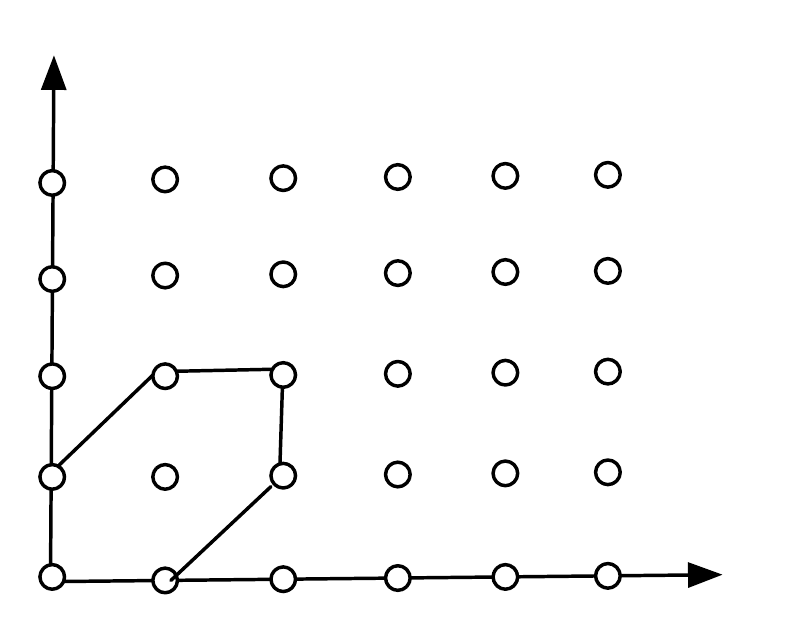}&\includegraphics[scale=0.25]{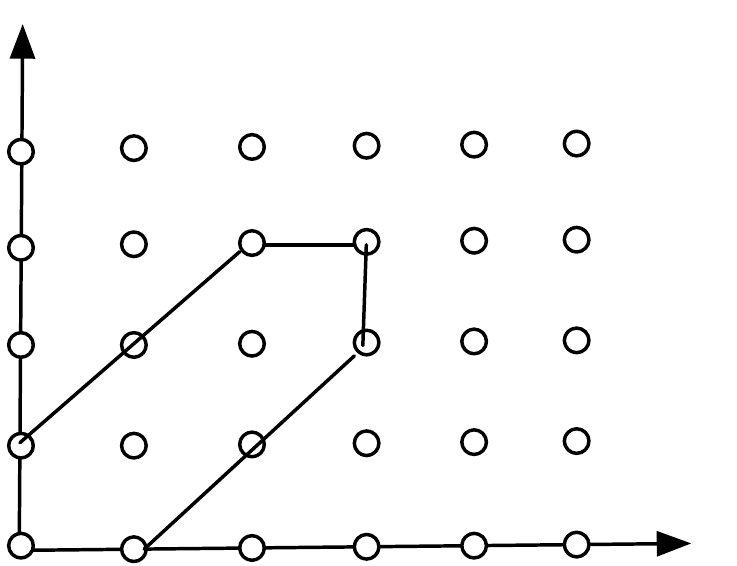}&\includegraphics[scale=0.25]{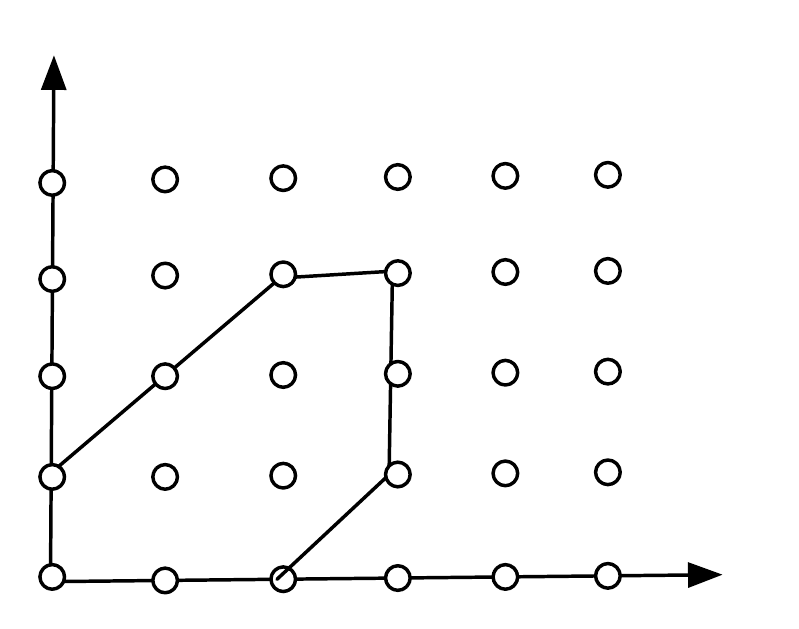}\\ \hline
\end{tabular}
\caption{All smooth 2-polytopes with at least 5 edges such that $|P\cap \Z^2|\le 12$}
\end{figure}
\end{lem}
\begin{proof}
See \cite{Lorenz}.
\end{proof}
We will need the following lemma which follows readily by the definition of smoothness.
\begin{lem}  \label{smoothfacets}  
Let $  P$   be a $n$-dimensional smooth polytope, then every facet of $  P$   viewed as subset of the supporting hyperplane containing it, is smooth.
\end{lem}  

\begin{lem}  \label{V14}  
Let $  P$   be a smooth 3-polytope such that $  |P\cap \Z^3|\le 16$, then $P$   has at most 8 facets and 12 vertices.
\end{lem}  

\begin{proof}  
Let $  P$   be a smooth 3-polytope such that $  |P\cap \Z^3|\le 16$, then by Lemma \ref{EFrest}   we see that  $  P$   has at most 10 facets.
Let $  P$   be a smooth 3-polytope with 9 facets and denote the facets by $  F_0,\dots ,F_8$. From Lemma \ref{EFrest}   it follows that $  P$   has 14 vertices. Since every pair of facets share at most 1 edge, no facet of $P$ may have more then 8 edges. Therefore there must be at least 14-8=6 vertices which do not lie in a given facet. Hence if $|P\cap \Z^3|\le 16$ then no facet of $P$ contains more then 10 lattice points. Therefore by Lemma \ref{Ngons} and Lemma \ref{smoothfacets} any facet of $P$ that is not triangular or quadrilateral must be either a pentagon or a hexagon.

Let $V(F_i)\in[3,6]$ be the number of vertices of the facet $F_i$. Since every vertex of $P$ lies in exactly three facets $\sum_{i=0}  ^8 V(F_i)=3\cdot 14=42$ holds. So by the pigeon hole principle there must be at least 3 facets such that $V(F_i)\ge 5$. Because every pentagonal or hexagonal facet has at least one interior lattice point this implies that $|P\cap \Z^3|\ge 17$.
A completely analogous argument for a 3-polytope with 10 facets establishes the lemma.\qed
\end{proof}  
 
The last piece of information we need in order to prove Proposition \ref{Odacoro} is Lemma \ref{3^24^36^2}.
\begin{lem}  \label{3^24^36^2}  
Up to isomorphism there exists exactly one smooth 3-polytope $  P$    that is associated to a triangulation with label $  3^24^36^2$   such that $  |P\cap \Z^3|\le 16$, namely the blow-up of $\Cay^2_\Sigma(3\Delta_1,\Delta_1, \Delta_1)$ at two vertices. This polytope has 7 facets and 10 vertices.
\end{lem}  
Lemma \ref{3^24^36^2} implies that if $P$ is a minimal smooth 3-polytope associated to a triangulation with label $3^24^36^2$ then $|P\cap\Z^3|>16$. Note that a smooth 3-polytope with a triangulation associated to the label $3^24^36^2$ has 7 facets and 10 vertices. These 10 vertices must all lie in a hexagonal facet since every pair of facets can share at most 2 vertices. The position of the hexagonal facets of  $P$ thus determines $P$.   In order to prove Lemma \ref{3^24^36^2} we first list the possible ways to align the smooth hexagons from the classification of Lorenz in the following Lemmata. This allows us to exhaust all possible positions of the hexagonal facets of a smooth polytope $P$ corresponding to a triangulation with label $3^24^36^2$ such that $|P\cap \Z^3|\le 16$. The following Lemmata follows by considering all possibilities.
\begin{lem}\label{orientBl3}  
The 2-polytope $  P= \conv((0,0),(1,0),(0,1),(1,2),(2,1),(2,2))$   is invariant under any isomorphism that takes a vertex to the origin and each edge containing the vertex to a coordinate axis.
\begin{center}
\begin{tabular}{|r|}
\hline
\includegraphics[scale=0.25]{Fig125-eps-converted-to.pdf}\\ \hline
\end{tabular}
\end{center}
\end{lem}  

\begin{lem}\label{orientBl32}  
Consider an isomorphism $\varphi$ that takes a vertex of the 2-polytope $  P=\conv((0,0),(1,0),(0,1),(2,1),(3,3),(3,4))$   to the origin and each edge containing the vertex to a coordinate axis. The image of $  P$ under $\varphi$   will be as illustrated in one of the 6 pictures below.
\begin{center}  
\begin{tabular}  {|r|r|r|}  
\hline
\includegraphics[scale=0.25]{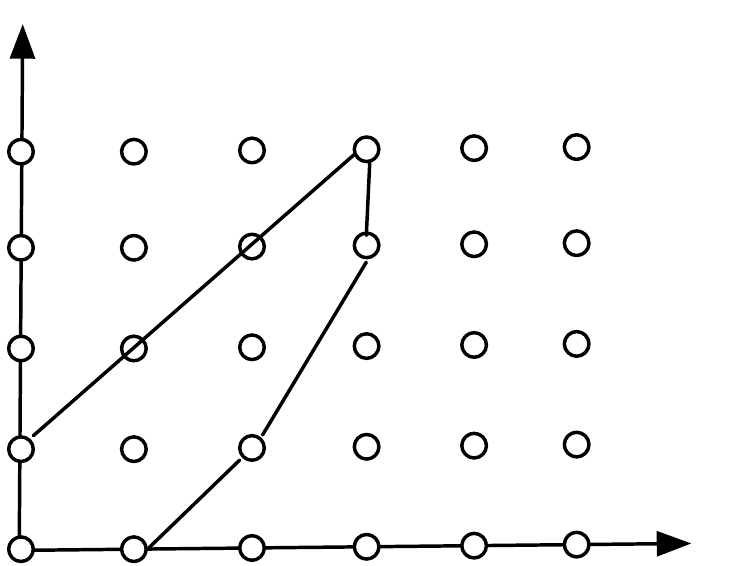}  & \includegraphics[scale=0.25]{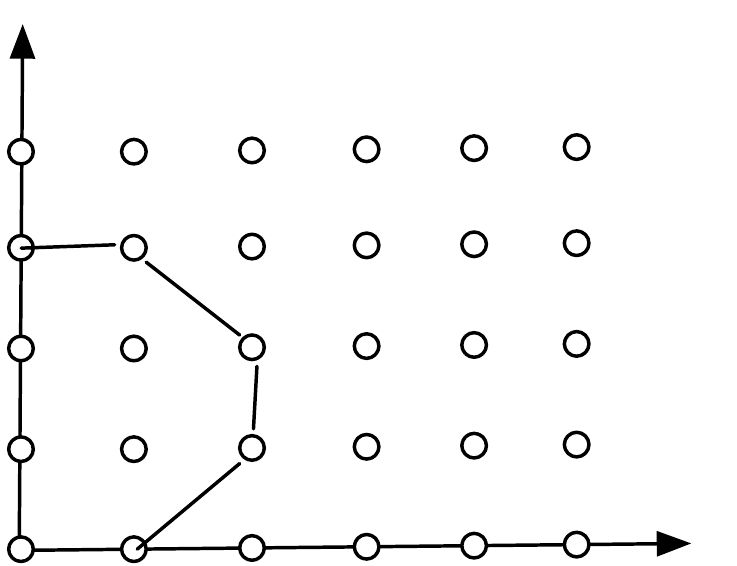}  & \includegraphics[scale=0.25]{Fig4-eps-converted-to.pdf}  \\ \hline
\includegraphics[scale=0.25]{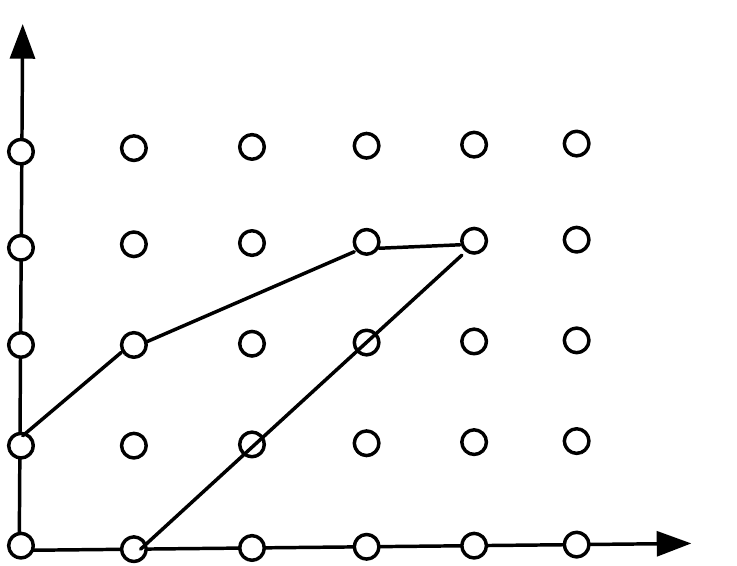}  & \includegraphics[scale=0.25]{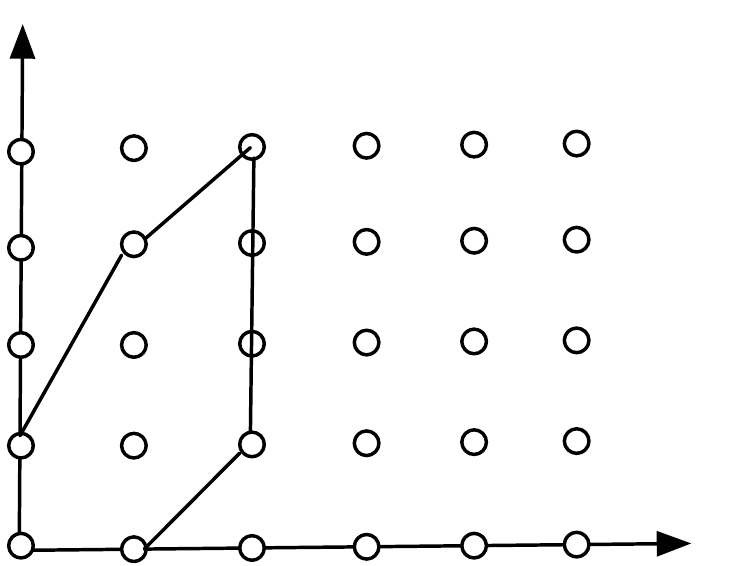}   & \includegraphics[scale=0.25]{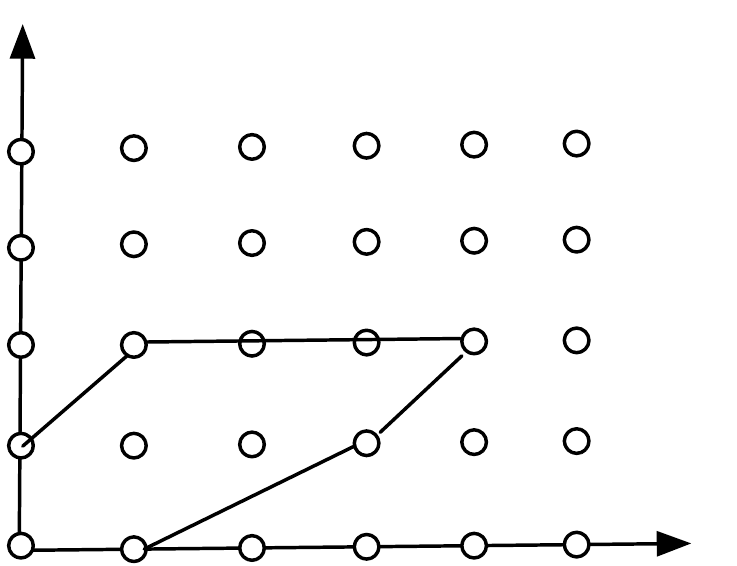}  \\ \hline
\end{tabular}  
\end{center}  
\end{lem}  

\begin{lem}  \label{orientBl33}  
Consider an isomorphism $\varphi$ that takes a vertex of the 2-polytope $  P=\conv((0,0),(1,0),(0,1),(2,3),(3,2),(3,3))$      to the origin and each edge containing the vertex to a coordinate axis. The image of $  P$ under $\varphi$ will be as illustrated in one of the pictures   below.
\begin{center}  
\begin{tabular}  {|r|r|r|}  
\hline
\includegraphics[scale=0.25]{Fig8-eps-converted-to.pdf}   & \includegraphics[scale=0.25]{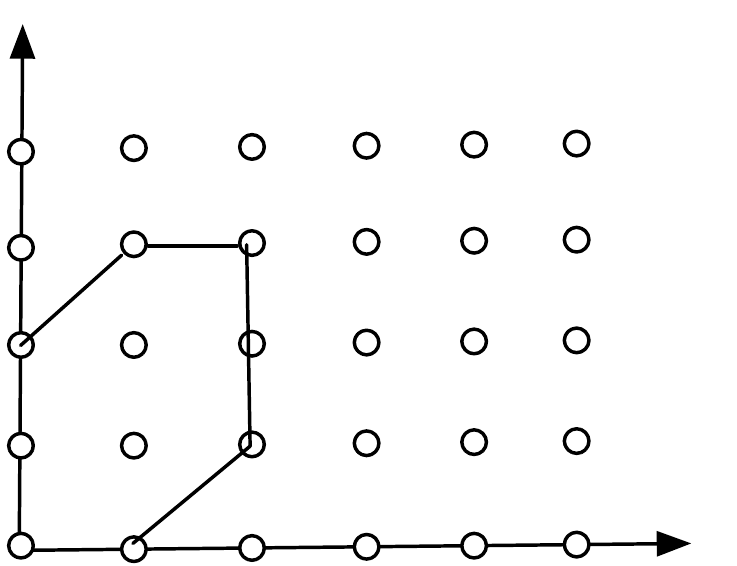}   & \includegraphics[scale=0.25]{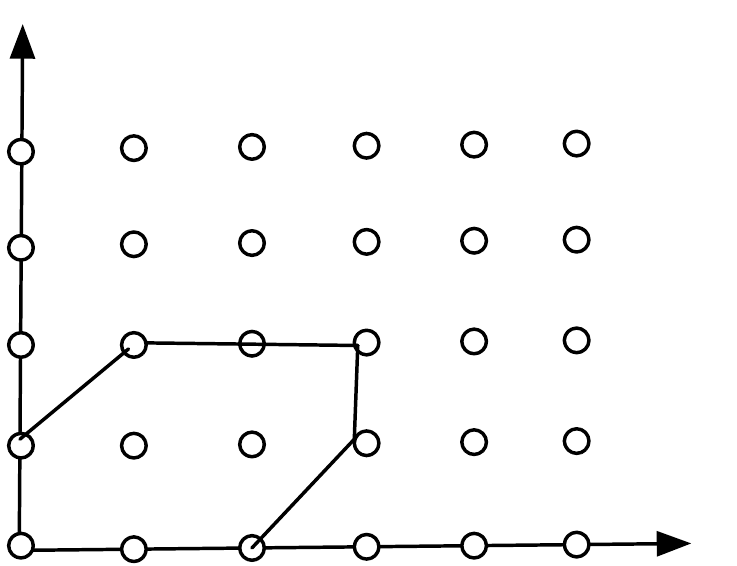}  \\ \hline
\end{tabular}  
\end{center}  
\end{lem}

\begin{proof}[of Lemma \ref{3^24^36^2}]
Let $  P$   be a smooth 3-polytope with $  |P\cap \Z^3|\le 16$      that is associated to a triangulation with label $  3^24^36^2$   and let $  F_1$   and $  F_2$   be the hexagonal facets of $  P$. Since $P$ has 10 vertices $  F_1$   and $  F_2$   must share an edge. Moreover $  F_1$   and $  F_2$   must be smooth 2-polytopes with respect to the hyperplane containing them by Lemma \ref{smoothfacets}. Therefore they must be isomorphic to some hexagonal 2-polytope given in the classification of \cite{Lorenz}. These are exactly the 2-polytopes appearing in Lemmas  \ref{orientBl3}, \ref{orientBl32},   and \ref{orientBl33}.

Note that because of smoothness we may, without loss of generality, assume that a shared vertex $v$ of $  F_1$   and $  F_2$   is positioned at the origin and that the edges through $  v$   are aligned along the coordinate axes in the positive direction. By the symmetry of the situation we may also assume that $  F_1$   and $  F_2$   lie in the $  xy$- and $  xz$-plane respectively with the shared edge of $  F_1$   and $  F_2$   along the $  x$-axis. However for any such configuration the points $  (1,2,0)$, $  (2,2,0)$, $  (1,0,2)$   and $  (2,0,2)$   lie in $  P$   by Lemmas \ref{orientBl3}  , \ref{orientBl32}   and \ref{orientBl33}. Then by convexity the points $  (1,1,1)$   and $  (2,2,2)$    also lie in $  P$. Assume that there are $  m$   lattice points on the shared edge of $  F_1$   and $  F_2$.  The position of every vertex of $  P$   is determined by the configuration of the hexagonal facets $  F_1$   and $  F_2$, so it must hold that $|P\cap \Z^3|\ge |F_1\cap \Z^3| +|F_2\cap \Z^3|+2-m.$

It is readily checked that since $m\le 4$ the only choice which allows $|P \cap  \Z^3|\le16$ is $F_1 \cong F_2 \cong \conv((0,0),(1,0),(0,1),(1,2),(2,1),(2,2))$. Since every vertex of $P$ is a vertex of $F_1$ or $F_2$ we see that $P$ is the blow-up of $  \Cay^2_\Sigma(3\Delta_1,\Delta_1, \Delta_1)$ at two vertices via the isomorphism $\varphi:\R^3\to \R^3$  defined by $  \varphi(x,y,z)=(2-x,y,z)$.\qed
\end{proof}

We are now in position to prove Proposition \ref{Odacoro}.
\begin{proof}[of Proposition \ref{Odacoro}]
We will prove the Proposition by computing lower bounds for the number of lattice points of $P$ for the labels of triangulations given in Theorem \ref{ODAthmrv}. Let $P$ be a smooth 3-polytope then $P$ has at most 8 facets by Lemma \ref{V14}. Since every edge of $P$ is shared by exactly two facets no facet can have more then $7$ edges. However if $F$ is a heptagonal facet of $P$ then $P$ Êhas 8 facets and 12 vertices by Lemma \ref{EFrest}. Thus exactly 5 vertices of $P$ Êdo not lie in $F$, so $F$ contains at most 11 lattice points since $|P\cap \Z^3|\le 16$. However by Lemma \ref{Ngons} Êthere exist no smooth heptagon with less then 13 lattice points. Hence every facet of $P$ Êis a smooth $n$-gon with $n\le 6$ and in particular $P$ cannot have a triangulation with label $3^24^47^2$.

From Lemma \ref{Ngons} we see that any pentagonal or hexagonal facet of $P$ has at least 1 interior point. Hence $P$ cannot have a triangulations with any of the labels $3^24^15^46^1, 3^14^35^46^1$ or $3^25^6$. Using Lemma \ref{Ngons} again we see that if $P$ has exactly 1 pentagonal facet, then that facet contains at least 3 lattice points apart from the vertices. So any smooth 3-polytope with a triangulation associated to the labels $3^34^15^16^3$ or $3^14^45^16^2$ contains at least 18 or 17 lattice points respectively.

Assume that $P$ contains exactly 2 pentagonal facets. Since the pentagonal facets can share their longest edge they contain at least 5 lattice points apart from the vertices. Hence $P$ cannot have a triangulation with label $3^24^25^26^2$. If $P$ has at least 3 pentagonal facets then a priori three different configurations are possible: i) At least 2 pentagonal facets do not share any edges. ii) Three pentagonal facets share an edge pairwise but these three pentagonal facets do not meet at any vertex. iii) Three pentagonal facets meet at a vertex. We will see that all three case are impossible if $P$  is smooth and $|P\cap \Z^3|\le 16$. Under these assumptions case i) is impossible since each pair of pentagonal facets $F$  and $F'$ that do not share any edges must contain at least 16 lattice points by Lemma \ref{Ngons} and there is at least one vertex of $P$, coming from a third pentagonal facet,  which does not lie in $F$  nor $F'$. Case ii) is impossible under the same assumptions, since a simple 3-polytope must have an even number of vertices by Lemma \ref{EFrest},  so $P$  has at least 10 vertices. Hence there must be at least 5 vertices which do not lie in any given pentagonal facet, so each such facet must be isomorphic to the one appearing in Lemma \ref{Ngons}. Therefore in the configuration of case ii) there is by exhaustion at least 8 lattice points apart from the vertices in $P$, which is a contradiction. For the final case note that three pentagonal facets contain a minimal number of lattice points if they are chosen and positioned in $P$ as illustrated below.
\begin{center}
\begin{figure}[h!]
\includegraphics[scale=0.5]{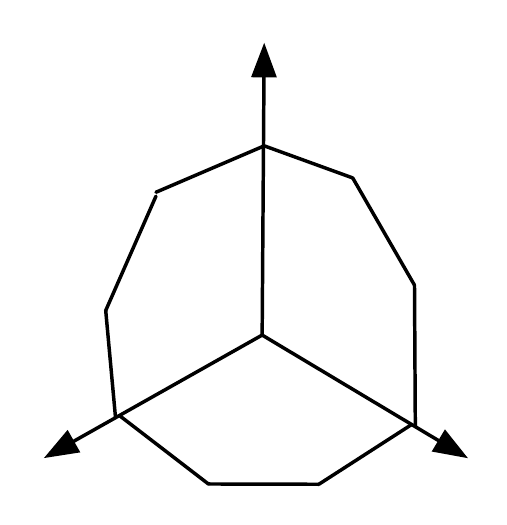}
\caption{Configuration of 3 pentagonal facets of $P$}
\end{figure}
\end{center}
In the configuration $(2,0,1)$ and $(0,2,1)$ lies in $P$ so therefore also $(1,1,1)=\frac{1}{2}(2,0,1)+\frac{1}{2}(0,2,1)$ lies in $P$ and is not a vertex. Therefore there are at least 7 lattice points of $P$ which are not vertices in this case. We can conclude that with our assumptions $P$ cannot contain three pentagonal facets, so in particular $P$ cannot be associated to a triangulation with label $3^14^35^3$ or $4^45^4$.
From the above and Theorem \ref{ODAthmrv} we conclude that if $P$ is minimal then it is associated to a triangulation with one of the labels $  3^4, 3^24^3, 4^6, 4^55^2, 4^66^2$    or $  3^24^36^2$. However if $P$ is minimal then the label of the triangulation associated to $P$ cannot be $  3^24^36^2$ by Lemma \ref{3^24^36^2}. Therefore it follows from Corollary \ref{Odacorollary} that if $P$ is a minimal smooth 3-polytope with $|P\cap \Z^3|\le 16$ then $P$ is strictly Cayley. \qed
\end{proof}   
\section{Some restrictions for the smoothness of strict Cayley 3-polytopes}\label{Orientsection}

Lemma \ref{sdivides}  and \ref{type2restriction}  provide us with enough restrictions to classify all smooth 3-polytopes $P$  such that $|P\cap \Z^3|\le 16$  and  $P\cong\Cay^s_\Sigma(P_0,P_1,P_2)$, where $P_0$, $P_1$  and $P_2$  are line segments. We will now establish similar restrictions for 3-polytopes of the type $P\cong \Cay^s_\Sigma(P_0',P_1')$,  were $P_0'$  and $P_1'$  are strictly isomorphic 2-polytopes. An obvious restriction is that we must have $  |P_0'\cap \Z^2|+|P_1'\cap\Z^2|\le 16$    which by the classification of smooth 2-polytopes in \cite{Lorenz}   implies that $  P_0'$ and $P_1'$ must correspond to $  \pn 2$, $\pn 2$Ê  $\Bl_2(\pn 2)$,  $  \Bl_3(\pn 2)$      or the Hirzebruch surface $\F_r$      where $  0\le r \le 4$. All the following lemmas follow from the definition of strict Cayley polytopes. 

\begin{lem}  \label{orientBl2rv} Consider an isomorphism $\varphi$ that takes a vertex of the 2-polytope $P:=\conv((0,0),(1,0),(0,1),(2,1),(2,3))$  to the origin and an edge  to each coordinate axes. The image of $P$ under $\varphi$ will be as illustrated in one of the figures below.
\begin{center}  
\begin{tabular}  {|r|r|r|r|r|}  
\hline
\includegraphics[scale=0.25]{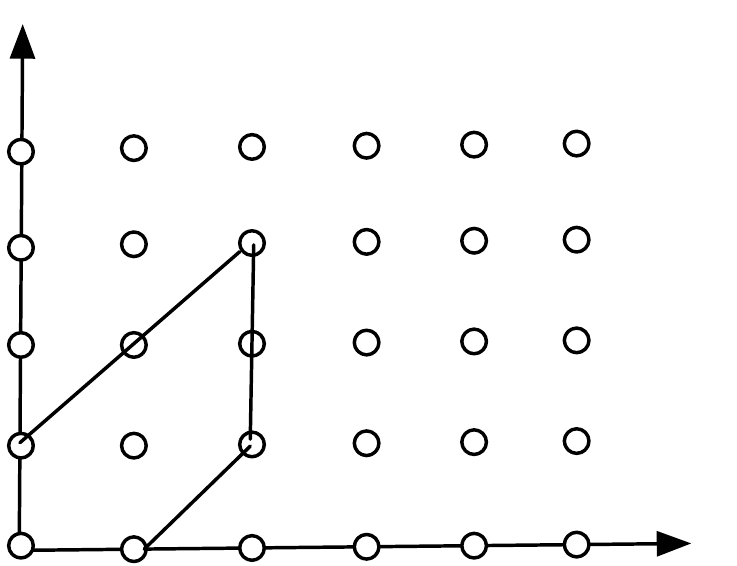}   & \includegraphics[scale=0.25]{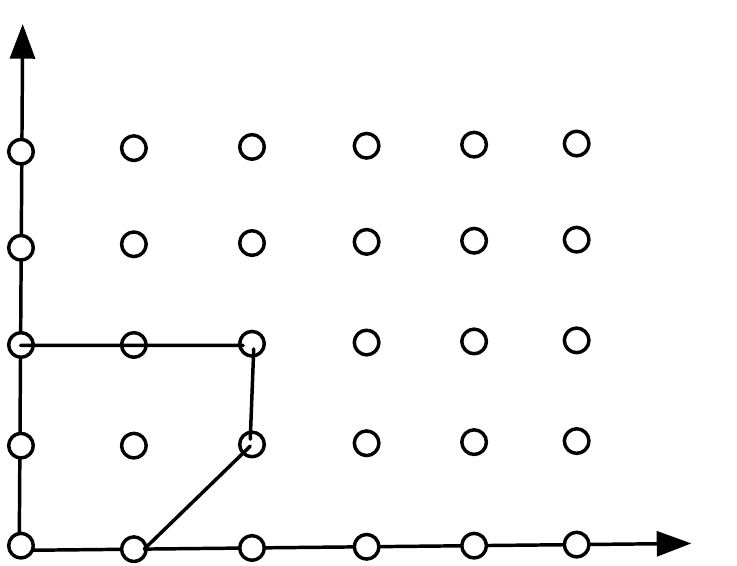}   & \includegraphics[scale=0.25]{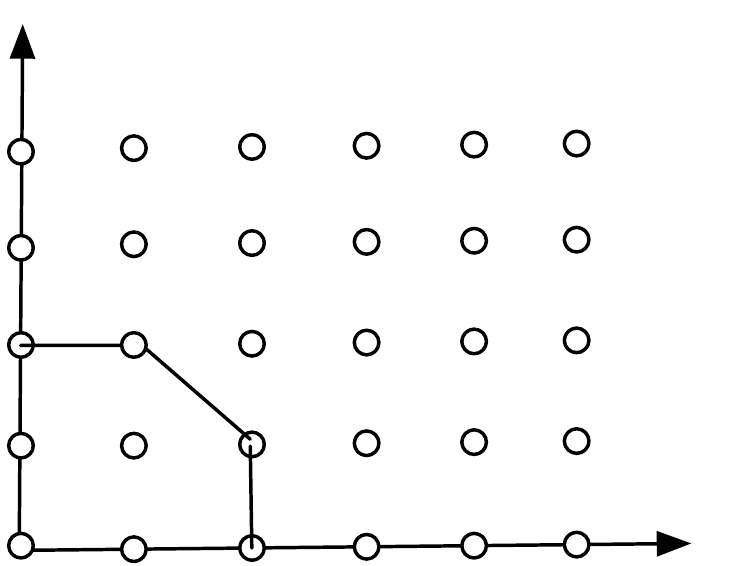}  &\includegraphics[scale=0.25]{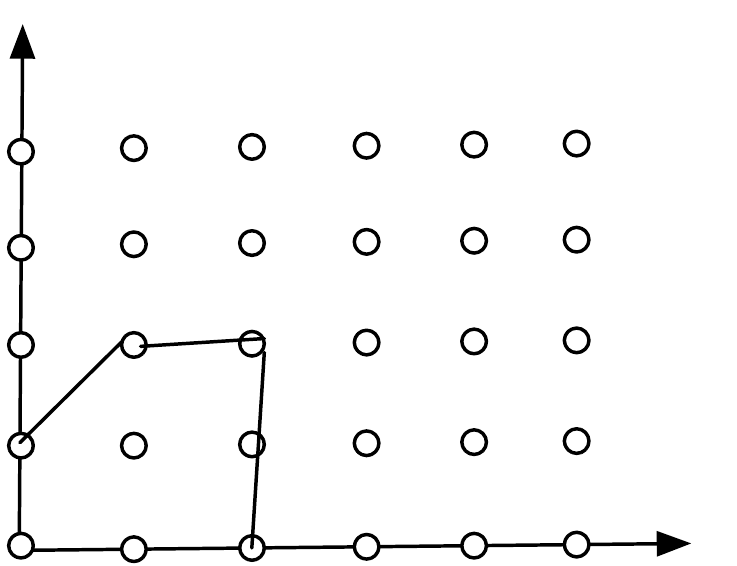}   & \includegraphics[scale=0.25]{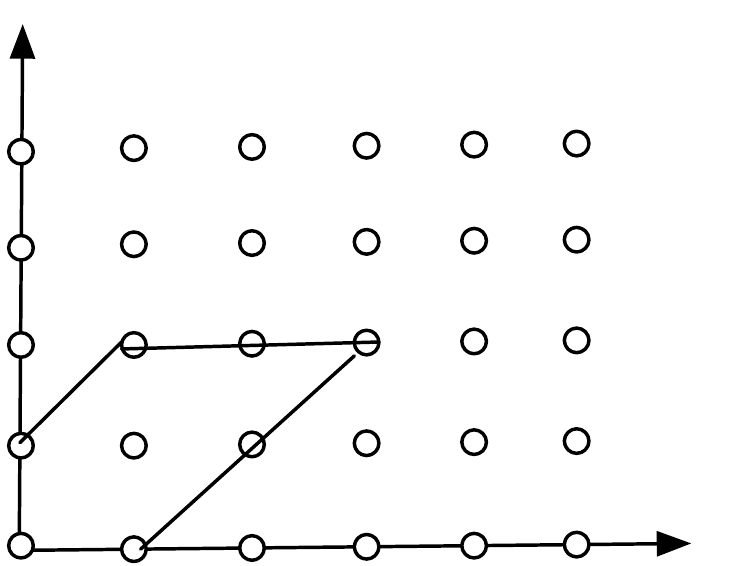}  \\ \hline
\end{tabular}  
\end{center}     
\end{lem}  

\begin{lem}\label{orientP2}\label{orientBl2}\label{orientF}\label{orientF0}
Let $P_0$  and $P_1$ be strictly isomorphic 2-polytopes with inner-normal fan $\Sigma$.
\begin{enumerate}
\item{If $\Sigma$  is the fan of $\pn 2$, $\F_r$ with $r\in\{1,\dots\}$ or the inner-normal fan of one of the 2-polytopes in Lemmas \ref{orientBl3}, \ref{orientBl32},  \ref{orientBl33} or \ref{orientBl2rv}, then there is exactly one polytope of the type $\Cay_\Sigma^s(P_0,P_1)$ for every $s\in \Z^+=\{1,\dots\}$.}
\item{If $\Sigma$ is the fan of $\pn 1\times \pn 1$ then $P_0\cong k_0\Delta_1\times k_0'\Delta_1$ and $P_1\cong k_1\Delta_1\times k_1\Delta_1$. Moreover if either $k_0=k_0'$ or $k_1=k_1'$ then then there is exactly one 3-polytope of the type $  \Cay^s_\Sigma(P_0,P_1)$ for each $s\in \Z^+$. If instead $k_0\ne k_0'$ and $k_1\ne k_1'$ then there is exactly two polytopes of the type $\Cay^s_\Sigma(P_0,P_1)$.}
\end{enumerate}
\end{lem}  

\begin{lem}\label{type3restriction}
Let $  P\cong \Cay^s_\Sigma(P_0,P_1)$      where $  P_0$   and $  P_1$      are strictly isomorphic 2-polytopes. If $  P$   is smooth and $  |P\cap \Z^3|\le 16$      then $  s\le 4$   and we may assume that $  |P_0\cap\Z^2|\ge|P_1\cap\Z^2|$. Moreover $  P$      is not a simplex and is isomorphic to a 3-polytope of the type $  \Cay^s_{\Sigma'}(P_0',P_1',P_2')$      if and only if $  P_0=P_1=k\Delta_2$   for some $  k\in \Z^+$.
\end{lem}  

\begin{proof}  
Let $P$ be oriented as in Lemma \ref{orientP2}. Because $P$ is convex, every lattice point of the form $(i,j,k)$ where $i,j\in\{0,1\}$ and $0\le k \le s$ lies in $P$. This proves the first part. For the second part note that all simplices have 4 vertices, all strict Cayley 3-polytopes of the type $  \Cay^s_{\Sigma'}(P_0',P_1',P_2')$      where $  P_0'$, $  P_1'$      and $  P_2'$      are line segments have 6 vertices. If the number of vertices of $  P_0$ and $P_1$   is $  k$   then the number of vertices of a 3-polytope of the type $  \Cay^s_\Sigma(P_0,P_1)$   is $  2k$. Therefore a 3-polytope of the type $  \Cay^s_\Sigma(P_0,P_1)$   is never isomorphic to a simplex and may only be isomorphic to a 3-polytope of the type $  \Cay^s_{\Sigma'}(P_0',P_1',P_2')$   if $  P_0\cong k_1\Delta_2$      and $  P_1\cong k_2\Delta_2$   for some $  k_1,k_2\in \Z^+$. If $  k_1\ne k_2$   then the 3-polytope $  \Cay^s_\Sigma(P_0,P_1)$   is a truncated pyramid. It is clear that such a 3-polytope does not contain three edges that are pairwise parallel and can therefore not be isomorphic to a 3-polytope of the type $  \Cay^s_{\Sigma'}(P_0',P_1',P_2')$. However if $  k_1=k_2=k$ then $\Cay^s_\Sigma(k\Delta_2,k\Delta_2)\cong \Cay^k_{\Sigma'}(s\Delta_1,s\Delta_1,s\Delta_1)$.\qed
\end{proof}

\section{A classification of all smooth strict Cayley 3-polytopes with at most 16 lattice points}\label{sectCay}

The following lemma lists all polytopes in category i) of Theorem \ref{CBconclusion} and follows by the definition of $k\Delta_3$.
\begin{lem}\label{type1}
$\Delta_3$   and $  2\Delta_3$   are the only smooth 3-simplices such that $|k\Delta_3\cap\Z^3|\le 16$.
\end{lem}  

The following lemma provides all polytopes in category ii) of Theorem \ref{CBconclusion}.
\begin{lem}\label{type2}
Up to isomorphism there are 69 smooth 3-polytopes $  P$      such that $  P\cong\Cay^s_\Sigma(P_0,P_1,P_2)$      where $  P_0=[0,i]$, $  P_1=[0,j]$   and $  P_2=[0,k]$      are line segments and $  |P\cap \Z^3|\le 16$. 
\end{lem}  

\begin{proof}  
From lemma    \ref{type2restriction}      we know that $s\le 2$. If $P\cong\Cay^2_\Sigma([0,i],[0,j],[0,k])$   then since we can assume $  i\ge j \ge k$      it follows from Lemma \ref{sdivides}   that we must have $  i-j=2m$   and $  i-k=2n$   where $  m,n\in \N$. Observe that $  P$   is completely determined by the choice of $  i$, $  m$      and $  n$ and that $  i>2n\ge 2m\ge0$. In particular the first three choices of $  i$, $  m$   and $  n$   give us the following smooth 3-polytopes for $  s=2$.
\begin{enumerate}[i)]
\item{$  i=1$      and $n=m=0$ gives $  P\cong\Cay^2_\Sigma(\Delta_1,\Delta_1,\Delta_1)$ and $  |P\cap \Z^3|=12$}  
\item{$  i=2$      and $n=m=0$ gives $ P\cong\Cay_\Sigma^2(2\Delta_1,2\Delta_1,2\Delta_1)$ and $  |P\cap \Z^3|=18$}  
\item{$  i=3$   and $n=m=1$ gives $P\cong\Cay^2_\Sigma(3\Delta_1,\Delta_1,\Delta_1)$ and    $  |P\cap \Z^3|=16$}  
\end{enumerate}  
For $  s=2$   all remaining choices of $  i$, $  m$      and $  n$      will clearly generate a smooth 3-polytope $  P$      with longer defining line segments than $  \Cay^2_\Sigma(3\Delta_1,\Delta_1,\Delta_1)$. Therefore all remaining choices will give $  |P\cap\Z^3|>16$. Hence we may conclude that $  \Cay^2_\Sigma(\Delta_1,\Delta_1,\Delta_1)$   and $  \Cay^2_\Sigma(3\Delta_1,\Delta_1,\Delta_1)$   are up to isomorphism the only smooth 3-polytopes of the type $  P\cong \Cay^2_\Sigma(P_0,P_1,P_2)$      such that $  |P\cap \Z^3|\le 16$. 

If $P\cong\Cay^1_\Sigma([0,i],[0,j],[0,k])$ we know from Lemma \ref{sdivides} that $  P$ is smooth. Therefore by Lemma    \ref{type2restriction}   the only restrictions are $  i\ge j \ge k$, $  i+j+k\le 13$   and $  k \le 4$. As readily can be checked there are in total 67 choices of $  i$, $  j$   and $  k$   meeting these restrictions, hence 67 associated smooth 3-polytopes. \qed 
\end{proof}  

This lemma provides the polytopes in category iii) of Theorem \ref{CBconclusion}.
\begin{lem}\label{type3}  
Up to isomorphism there are 33 smooth 3-polytopes $  P$   of the type $  P\cong \Cay^s_\Sigma(P_0,P_1)$   where $  P_0$   and $  P_1$   are strictly isomorphic 2-polytopes such that $  |P\cap\Z^3|\le 16$. Of these 5 are isomorphic to a polytope counted in Lemma \ref{type2}. 
\end{lem}
\begin{proof}
Consider the case when $  P_0$   and $  P_1$   are smooth  2-simplices. By Lemma \ref{type3restriction} we know that $  s\le 4$. For a fixed choice of $  P_0$, $  P_1$   and $  s$, Lemma \ref{orientP2} implies that  there is up to isomorphism at most one smooth 3-polytope of the type $  \Cay^s_\Sigma(P_0,P_1)$. In order for a 3-polytope $  P$   to fit our classification it must hold both that $  |P_0 \cap \Z^2|+|P_1\cap \Z^2|\le |P \cap \Z^3| \le 16$   and that  every edge between $(P_0,0)$ and $(P_1,s)$ contains exactly $s+1$ lattice points by Lemma \ref{type3restriction2}. If $  s=4$      then the choice $  P_0\cong P_1 \cong \Delta_2$   gives the 3-polytope $  P\cong \Cay^4_\Sigma(\Delta_2,\Delta_2)$, for which $|P\cap   \Z^3|= 15$. Since $  2\Delta_2$   contains 3 lattice points more then $  \Delta_2$ it follows that $  \Cay^4_\Sigma(\Delta_2,\Delta_2)$ must be the only smooth 3-polytope of the type $P\cong \Cay^4_\Sigma(k_1\Delta_2, k_2\Delta_2)$. By considering all possible choices when $  s=3,2$ and $1$   in the same way we get 7 more smooth 3-polytopes meeting our criteria. These are all 3-polytopes of the type $  P\cong\Cay^s_\Sigma(k_1\Delta_2,k_2\Delta_2)$ such that $|P\cap \Z^3|\le 16$. None of the 8 smooth 3-polytopes  considered so far in the proof are isomorphic to each other and using Lemma \ref{type3restriction} we see that only those that are isomorphic to triangular prisms are isomorphic to a polytope covered in Lemma \ref{type1}  or Lemma \ref{type2}.

In the same way one may utilize Lemma \ref{orientF}   to show that there are 1, 2, 4 and 9 smooth 3-polytopes of the form $  \Cay^s_\Sigma(P_0,P_1)$   where $  P_0$   and $  P_1$   have an inner-normal fan associated to $  \F_4$, $  \F_3$, $  \F_2$   and $  \F_1$   respectively. Similarly, one can check that there is 1 smooth 3-polytope $  P$   of the type $  \Cay^s_\Sigma(P_0,P_1)$   both when $  P_0$   and $  P_1$   have inner-normal fan associated to $  \Bl_2(\pn 2)$   and when the inner-normal fan of $  P_0$   and $  P_1$   is associated to $  \Bl_3(\pn 2)$. If the inner-normal fan of $  P_0$      and $  P_1$      is associated to $ \F_0$      then $16\ge |P\cap \Z^3|\ge |P_0\cap \Z^2|+|P_1\cap \Z^2|$, $|P_0Ê\cap \Z^2|\ge 4$ and $|P_1\cap \Z^2|\ge 4$. Therefore $P_0$ and $P_1$ contains at most 12 lattice points each and by the classification in \cite{Lorenz} we have up to isomorphism 12 possible choices for $  P_0$   and $  P_1$. Lemma \ref{orientF0} implies that for a choice up to isomorphism such that neither $  P_0$   nor $  P_1$ is a square there are at most 2 smooth 3-polytopes of the type $\Cay^s_\Sigma(P_0,P_1)$. Taking this into account one may proceed completely analogously to when $P_0$ and $P_1$ are simplices using Lemma \ref{orientF0}   and Lemma \ref{type3restriction2}. Such a procedure gives up to isomorphism 16 smooth 3-polytopes of the type $  \Cay^s(P_0,P_1)$   where $  P_0$   and $  P_1$   have inner-normal fan associated to $  \F_0$. These last 16 smooth 3-polytopes appear in the appendix as  the 9 polytopes corresponding to $\pn 1$-bundles over $\F_0$ and 7 prisms which already have been considered as a $\pn 1$-bundles over $\F_r$ for some $r>0$. ÊThe lemma now follows after we have excluded one polytope from every pair of isomorphic polytopes in our list.\qed
\end{proof}


\section{Classification of smooth 3-polytopes associated to blow-ups}  \label{sectListBl}  
\begin{lem}\label{needBl}
Any smooth non-minimal 3-polytope $P$   such that $  |P\cap \Z^3|\le 16$   can be obtained by consecutive  blow-ups of  a polytope associated to a triangulation with one of the labels: $3^4,3^24^3, 4^6,4^55^2,3^24^36^2 \text{ or }      3^14^35^3$.
\end{lem}  
\begin{proof}
By Lemma \ref{V14}    any smooth 3-polytope $  P$   such that $  |P\cap \Z^3|\le 16$   has at most 8 facets. Note that every blow-up of a 3-polytope $  P$   adds a facet to $  P$. Therefore if $  P$      is a minimal smooth 3-polytope then a necessary condition for the $n$:th consecutive  blow-up of $P$ to contain at most 16 lattice points is that $  P$      has at most $  8-n$   facets. Moreover every non-minimal and smooth 3-polytope may be obtained by consecutive blow-ups of some minimal and smooth 3-polytope \cite{Ewald}. Theorem \ref{ODAthmrv}  now establishes the lemma.\qed\end{proof}

By definition we may construct a new polytope $  \Bl_F^{(k)}(P)$     by blowing-up a 3-polytope $  P$  either at a vertex or along an edge. Blowing up along an edge will add a quadrilateral facet to the 3-polytope and blowing up a vertex will add a triangular facet to the 3-polytope. When a vertex is blown-up the three facets meeting in that vertex will each gain one more edge. This implies that the label of $  \Bl_v^{(k)}(P)$    may be obtained by adding a '3' and raising 3 base numbers in the label of $  P$. Consider for example the case when $P$ is associated to a triangulation with label $4^55^2$. If we blow-up $P$ at a vertex where one pentagonal and 2 quadrilateral facets meet, then $\Bl_v^{(k)}(P)$ will have a triangulation with label $3^14^35^36^1$ since a triangular facet is added and each facet containing the vertex gets one more edge. Similarly if we blow-up along  an edge $  e$  then the two facets that have the end-points of $  e$    as a vertex but do not contain the edge $  e$      will each gain one more edge. Hence the label of $  \Bl_e^{(k)}(P)$     may be obtained from the label of $  P$   by adding a '4' and raising 2 base numbers. 

\begin{lem}  \label{Bl3^14^35^3}  
There exists no smooth 3-polytope $P$   such that
\begin{enumerate}[i)]
\item{$  |P   \cap \Z^3|\le 16$  }  
\item{$P$ can be obtained by blow-ups starting from a 3-polytope associated to a triangulation with label $  3^14^35^3$, $  4^6$, $3^24^36^2$   or $  4^55^2$  }  
\item{$P$   has not been counted in Lemma \ref{3^24^36^2}, \ref{type1} or \ref{type2}, \ref{type3}.}  
\end{enumerate}  
\end{lem}  

\begin{proof}Remember that a smooth polytope $P$  such that $|P\cap \Z^3|\le 16$  has at most 8 facets. Hence it is enough to consider polytopes obtained via a single blow-up of a  polytope associated to a triangulation with label $  3^14^35^3$, $3^24^36^2$   or $  4^55^2$ and polytopes obtained via at most two blow-ups of a polytope associated to a triangulation with label $4^6$.
Given a 3-polytope $  P$      with label $  3^14^35^3$   we list all the ways one can add 1 to the exponent of 3 in the label and raise three of the base numbers. This list includes the label associated to every possible blow-up $  \Bl_v^{(k)}(P)$ at a vertex $  v$   of $  P$. For each label we can check if it corresponds to a triangulation of $  S^2$      listed in \cite[A5]{Oda}   i.e. if it can be associated to a simple 3-polytope. Moreover we can compute a lower bound for $  |\Bl_v^{(k)}(P)\cap \Z^3|$      by using Lemma \ref{Ngons}   and considering the number of lattice points in the facets of $  \Bl_v^{(k)}(P)$. Certainly any  blow-up of $  P$ that fits our classification     needs to be both associated to a simple 3-polytope and have at most 16 lattice points in its facets. 

To compute a lower bound for the number of lattice points in the facets we use the same technique as in the proof of Proposition  \ref{Odacoro}. Consider for example the case when $P$ is associated to a triangulation with  label $3^14^35^3$ and we blow-up $P$ at a vertex $v$ where the triangular facet $F_3$ and 2 of the quadrilateral facets $F_4, F_4'$ meet. Then $P$ will gain a triangular facet while $F_3,F_4,F_4'$ all gain one more edge. Hence $\Bl_v^{(k)}(P)$ will be associated to a triangulation with label $3^14^25^5$. To get a lower bound for $|\Bl_v^{(k)}(P)\cap \Z^3|$ note that $\Bl_v^{(k)}(P)$Ê contains 5 pentagonal facets and these pairwise can share at most one edge. Thus by the proof of Proposition \ref[A5]{Odacoro} $\Bl_v^{(k)}(P)$ must have at least 9 lattice points in its facets that are not vertices. Because $\Bl_v(P)$ have 8 facets it has 12 vertices by Lemma \ref{V14} hence $|\Bl_v^{(k)}(P)\cap \Z^3|\ge 21$. Next we note that the label $3^14^25^5$ does not appear in the classification in \cite[A5]{Oda}. Thus in this case $\Bl_v^{(k)}(P)$ fails both of our criteria. The results of such a calculation for all blow-ups at a vertex of a polytope $P$ with label $3^14^35^3$  are given below. 
\begin{center}  
\begin{tabular}  {|r|r|r|}  
\hline
Label of $  \Bl_v^{(k)}(P)$     & Associated to a triangulation & Lower bound\\ \hline
$  3^14^25^5$   &  No & $  |\Bl_v^{(k)}(P)\cap \Z^3|\ge 21$  \\ \hline
$  3^14^35^36^1$   & Yes& $  |\Bl_v^{(k)}(P)\cap \Z^3| \ge 19$  \\  \hline
$  3^14^45^16^2$   & Yes& $  |\Bl_v^{(k)}(P)\cap \Z^3| \ge 17$  \\  \hline
$  3^25^6$   & No& $  |\Bl_v^{(k)}(P)\cap \Z^3| \ge 22$  \\ \hline
$  3^24^15^46^1$   & Yes &$  |\Bl_v^{(k)}(P)\cap \Z^3| \ge 21$  \\ \hline
$  3^24^25^26^2$   & Yes &$  |\Bl_v^{(k)}(P)\cap \Z^3| \ge 19$  \\  \hline
$  3^24^36^3$   & No & $  |\Bl_v^{(k)}(P)\cap \Z^3| \ge 15$  \\ \hline
\end{tabular}
\end{center}
From the above list it is clear that there exist no smooth polytopes of the type $\Bl_v^{(k)}(P)$, where $P$  is a polytope with label $3^14^35^3$, such that $|\Bl_v^{(k)}(P)\cap \Z^3|\le 16$ . 

If we blow-up along an edge $  e$   of a 3-polytope $  P$      with label $  3^14^35^3$   and then use the same technique as above a similar list including the label of every possible such blow-up can be created. From that list we then conclude that there exist no smooth polytopes of the type $  \Bl_e^{(k)}(P)$, where  $P$ is a polytope with label $3^14^35^3$ such that $  |\Bl_e^{(k)}(P)\cap \Z^3|\le 16$. Applying the same technique to blow-ups of 3-polytopes with the labels $  3^24^36^2$, $  4^6$   and $  4^55^2$   we see that there exist no blow-ups meeting all three stated criteria. In particular condition \textit{iii)} is used when we obtain blow-ups with the label $4^6$ when blowing up a polytope with label $3^24^36^2$ along an edge. \qed
\end{proof}  

\begin{lem}  \label{Bl3^4n3^24^3}  
Up to isomorphism there are 3 smooth 3-polytopes $  \Bl(P)$,   such that $  |\Bl(P) \cap \Z^3|\le 16$ which are not counted in Lemma \ref{3^24^36^2}, \ref{type1}, \ref{type2}, \ref{type3} or \ref{Bl3^14^35^3}. These can be obtained by a sequence of blow-ups of a 3-polytope $  P$   with one of the labels $  3^4$     or $  3^24^3$. The three polytopes are the blow-up at a vertex of $\Cay^2_\Sigma(3\Delta_1,\Delta_1,\Delta_1) $, the blow-up at 4 vertices of $ 3\Delta_3$ and the blow-up at 2 vertices of $\Cay^2_\Sigma(2\Delta_1,2\Delta_1,2\Delta_1)$.
\end{lem}  

\begin{proof}  
For a smooth 3-polytope associated to the triangulation $  3^24^3$   one triangular facet meets two quadrilateral facets at each vertex. Therefore if we blow-up along an edge joining the two triangular facets we get a 3-polytope with the label $  4^6$   i.e. only already considered 3-polytopes. However if we blow-up a vertex or an edge of a triangular facet then we get a 3-polytope associated to the label $  3^24^25^2$. 

From the classification of all triangulations of $  S^2$   given in \cite{Oda}   we see that two pentagonal facets of a 3-polytope $  P$      associated to the triangulation $3^24^25^2$   will share an edge. By the classification of smooth 2-polytopes given in \cite{Lorenz}   we see that the two pentagonal facets of $  P$      must be isomorphic to the 2-polytope $ F:=\conv((0,0),(1,0),(0,1),(2,1),(2,3))$ since $P$  has $8$  vertices and there must be at least one interior lattice point in each pentagonal facet by smoothness and convexity. Note that every vertex of a 3-polytope associated to triangulation with label $  3^24^25^2$  lies in a pentagonal facet. Hence all 3-polytopes with that label may be formed by taking the convex hull of every permissible configurations of the two pentagonal facets. To do this start by choosing 2 of the alignments of $  F$ given  in Lemma \ref{orientBl2rv}   with the same edge length along the $  x$-axis. Position one in the $  xy$-plane and one in the $  xz$-plane via the maps $(x,y)\mapsto (x,y,0)$ and $(x,y)\mapsto (0,x,y)$. It is readily checked that exploring every such possibility gives us up to isomorphism exactly 1 smooth 3-polytope with at most 16 lattice points, namely the blow-up of $\Cay^2_\Sigma(3\Delta_1,\Delta_1,\Delta_1) $ at a vertex.  

Since a 3-polytope associated to a triangulation with label $  3^24^25^2$   only has 6 facets it may be blown-up again. Therefore we apply the techniques of the proof of Lemma \ref{Bl3^14^35^3} to a polytope $P$ associated to a triangulation with label $3^24^25^2$. This reveals that all blow-ups of $P$ such that $|P\cap \Z^3|\le 16$ have already been considered or have a triangulation with label $3^24^25^26^1$. If $\Bl_F^{(k)}(P)$ have a triangulation with label $3^24^25^26^1$ then by the proof of Proposition \ref{Odacoro} there are at least 16 lattice points in the facets of $\Bl_F^{(k)}(P)$. When $\Bl_F^{(k)}(P)$ have exactly 16 lattice points the pentagonal and hexagonal facets must be the smallest possible and the pentagonal facets must share an edge of length 2. Since $\Bl_F^{(k)}(P)$ have 7 facets in total the hexagonal facet $F_6$ must share an edge with every other facet. Moreover if we fix a pentagonal facet $F_5$ then 9 of the 10 vertices must lie in either $F_6$ or $F_5$ while the remaining vertex lies in the other pentagonal facet $F_5'$. Thus we can list all polytopes $\Bl_F^{(k)}(P)$ associated to a triangulation with label $3^24^25^26^1$ meeting our criteria by finding all possible ways to position $F_5,F_5'$ Êand $F_6$ so that the following holds
\begin{enumerate}
\item{$F_6$ lies in the $xy$-plane with a vertex at the origin and an edge along each coordinate axis.}
\item{$F_5$ lies in the $xz$-plane, sharing an edge with $F_6$ and has one edge along the $z$-axis.}
\item{$F_5'$ shares one edge of length 2 with $F_5$ and one edge with $F_6$. }
\item{Let $\Bl_F^{(k)}(P)$ be the convex hull of the configuration, then $\Bl_F^{(k)}(P)$ is smooth and $|\Bl_F^{(k)}(P)\cap \Z^3|\le 16$}
\end{enumerate}
The three possible positions of $F_5$ and $F_6$ that meet requirement 1 and 2 are easily obtained from Lemma \ref{orientBl2rv} and are illustrated below.
\begin{center}
\includegraphics[scale=0.7]{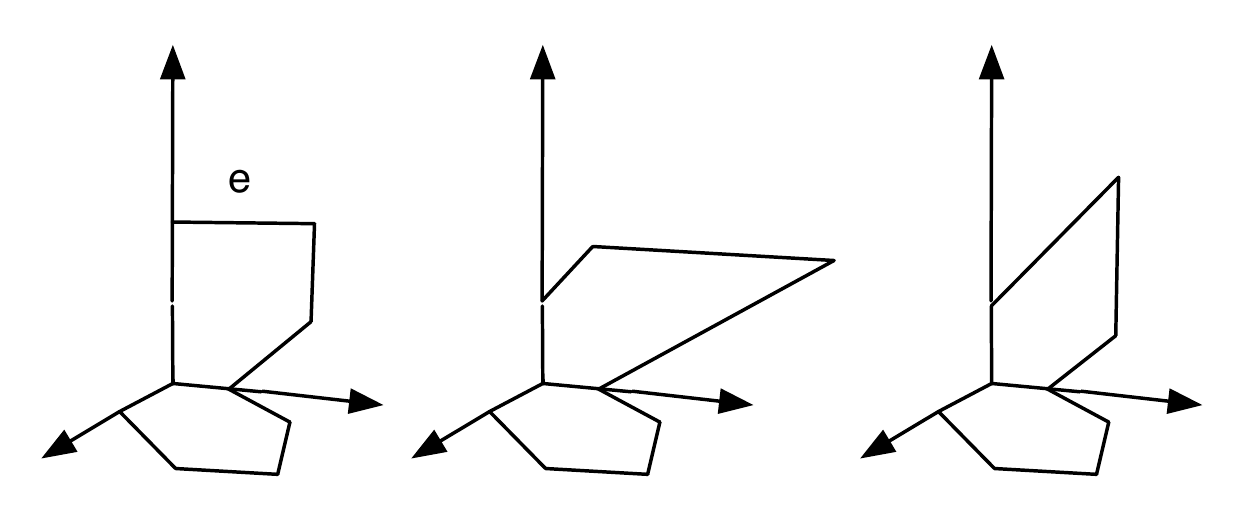}
\end{center}
Note that  the convex hull of the left configuration is not smooth at the vertex $(0,1,0)$. This reduces the possible ways to align $F_5'$ into it either being contained in the $yz$-plane or sharing the edge $e$ of length 2 with $F_5$ Êand having a vertex at $(0,1,1)$. The first configuration contains the point $(1,1,2)=\frac{1}{2}(2,0,2)+\frac{1}{2}(0,2,2)$ so in this case $|\Bl_F^{(k)}(P)\cap \Z^3|>16$. The second configuration is $\Cay^2_\Sigma(2\Delta_1,2\Delta_1,Ê2\Delta_1)$ blown-up at two vertices. Note that the blow-up of $\Cay^2_\Sigma(2\Delta_1,2\Delta_1,2\Delta_1)$ at two vertices is a smooth 3-polytope that contains exactly 16 lattice points. By the same argument we see that in the middle configuration $\Bl_F^{(k)}(P)$ must have a vertex at $(3,1,1)$. The resulting polytope in this case is easily seen to be isomorphic to $\Cay^2_\Sigma(2\Delta_1,2\Delta_1,2\Delta_1)$ blown-up at two vertices. In the right configuration we see that $\Bl_F^{(k)}(P)$ must have a vertex at $(2, 2, 1)$ i.e. $F_5'$ lies in the plane $x=2$. However $\Bl_F^{(k)}(P)$ also contains the lattice point $(1,1,1)=\frac{1}{2}(1,0,2)+\frac{1}{2}(1,2,0)$ which is not a vertex so $|\Bl_F^{(k)}(P)\cap \Z^3|>16$ in this case. 

Next, we need to consider all smooth 3-polytopes that can be obtained via 2 blow-ups of a 3-polytope with a triangulation associated to the label $  3^24^25^2$. This time the techniques of Lemma \ref{Bl3^14^35^3} reduce the blow-ups we need to consider to solely those associated with triangulation having the label $  3^46^4$. 

From the classification in \cite[A5]{Oda}   each pair of the four hexagonal faces of a 3-polytope $\Bl(P)$ with label $  3^46^4$   share exactly one edge. Moreover all hexagonal facets must be isomorphic to $F:=\conv((0,0),(1,0),(0,1),(1,2),(2,1),(2,2))$   since otherwise $\Bl(P)$   will contain  more then 16 lattice points in its facets. Note that using Lemma \ref{orientBl3}   we can assume that two of the hexagonal facets lie in the $  xy$- and $  yz$-plane as illustrated in the figure to the  left below. By the triangulation given in \cite[A5]{Oda}   the facet in the $  xz$-plane must be one of the other two hexagonal facets. Since every hexagonal facet has to be isomorphic to $  F$   this can only be done as illustrated  in the second figure below. This procedure determines all 12 vertices of $\Bl(P)$. Therefore the polytope is obtained by taking the convex hull of the configuration, as indicated in the right figure.
\begin{figure}[h] 
\includegraphics[scale=0.4]{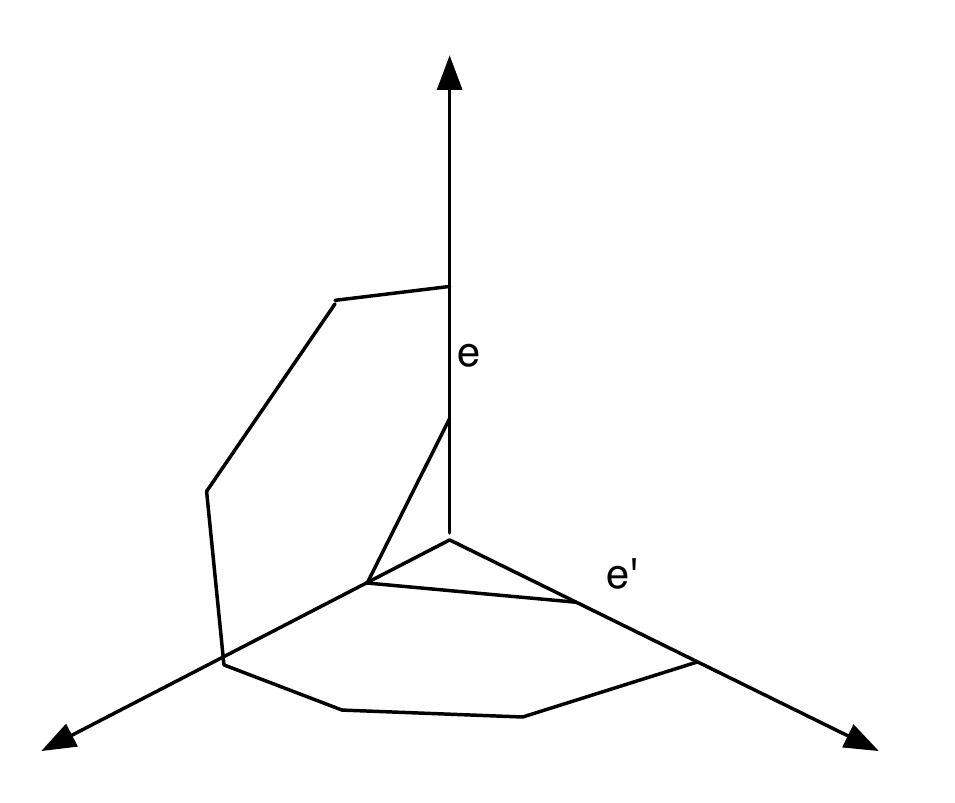}  \includegraphics[scale=0.4]{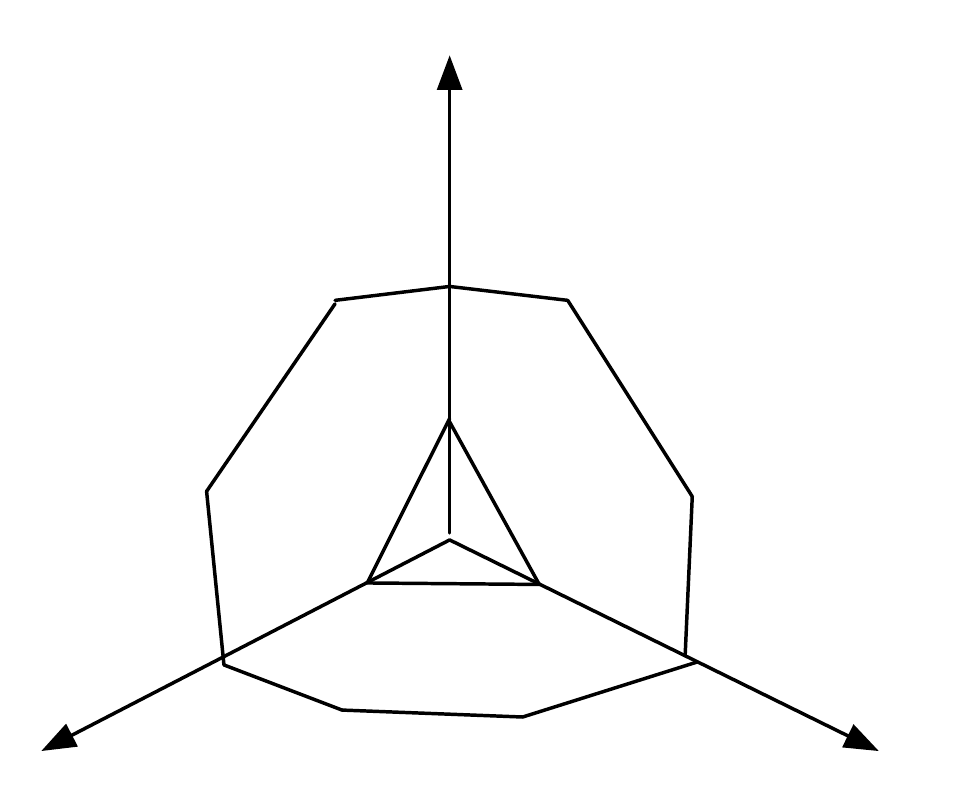}  \includegraphics[scale=0.5]{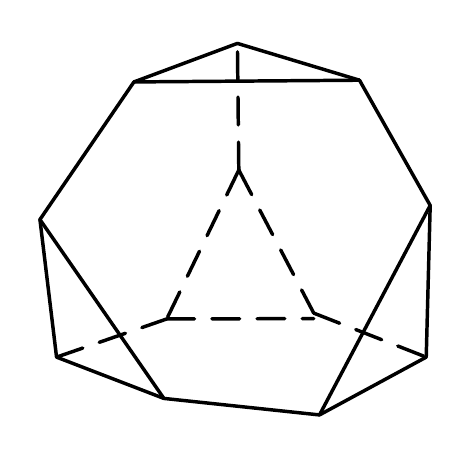}  
\end{figure}  
We note that $\Bl(P)$   is the blow-up of $  3\Delta_3$ at four vertices, that $  |\Bl(P)\cap \Z^3|=16$ and that this is the only 3-polytope with label $  3^46^4$   meeting our requirements. Finally if we blow-up of a simplex at a vertex or along an edge we get a  3-polytope with label $  3^24^3$. Hence all blow-ups of simplices have already been considered.\qed
\end{proof}

The above lemma covers the last cases to be considered by Lemma \ref{needBl}. Theorem \ref{CBconclusion}  now follows and our classification is complete. A complete library of all polytopes and embeddings in our classification is given in the appendix. Finally we briefly mention how our results relate to two conjectures in toric geometry. 
  
\section{Smoothness and projective normality}
The results of this paper relates to the following two conjectures in toric geometry.
\begin{conjecture}\label{conj2}
Every smooth toric variety has a quadratic Gr\"{o}bner basis.
\end{conjecture}

\begin{conjecture}\label{conj1}
Every smooth toric variety is projectively normal.
\end{conjecture}

Remember that a projective toric variety $X$  is projectively normal if the affine cone of $X$  is a normal variety. 

We have checked both these conjectures using the library \textsf{toric.lib} in \textsf{Singular} \cite{singular}. The reduced Gr\"{o}bner basis with respect to the degree reverse lexicographical ordering consists entirely of quadratic binomials for every toric ideal $I_{P\cap \Z^3}$ Êcorresponding to a polytope in the classification. This means that Conjecture \ref{conj2} holds for all polytopes in our classification. To check Conjecture  \ref{conj1} we have used the following Proposition.
\begin{prop}
Let $\mathcal{A}Ê\subset \Z^d$ define a homogenous toric ideal $I_\A$. Assume that $\prec$ is a term ordering on   $\C[x_1\dots,x_n]$ and that the initial ideal $in_\prec(I_\A)$ with respect to $\prec$ is square free. Then the projective toric variety $X_\A$ associated to $I_\A$ is projectively normal.
\end{prop}
\begin{proof}
See \cite[p.136]{sturmfels}.
\end{proof}
For the ideals corresponding to the 3-polytopes in the classification the initial ideal $in_{lex}(I_\A)$ with respect to the lexicographical term ordering is square free. So in particular conjecture \ref{conj1} holds for all 3-polytopes in the classification. In combinatorial terms this means that the placing triangulation is a regular unimodular triangulation for all 3-polytopes in the classification (see \cite[p.67]{strumfels91}). A hierarchic list of properties implying projective normality can be found in \cite[p.2313]{oberwolfach}. Having a regular unimodular triangulation is the strongest such property that holds for all polytopes in our classification. For more details on how these computations are done see: http://www.math.kth.se/$\sim$alundman.
\newpage
\section{Appendix: A complete list of all smooth lattice 3-polytopes $P$ such that $|P\cap\Z^3|\le 16$.}\label{app}
\begin{center}
\begin{tabular}{|p{3.4cm}|p{3.4cm}|p{3.4cm}|}
\hline
\includegraphics[scale=0.7]{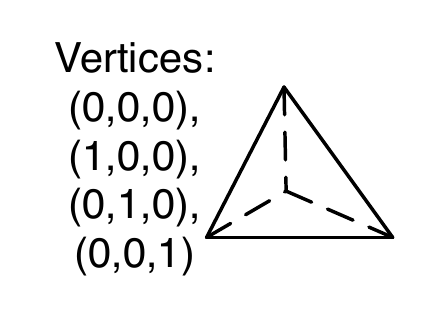}& \includegraphics[scale=0.7]{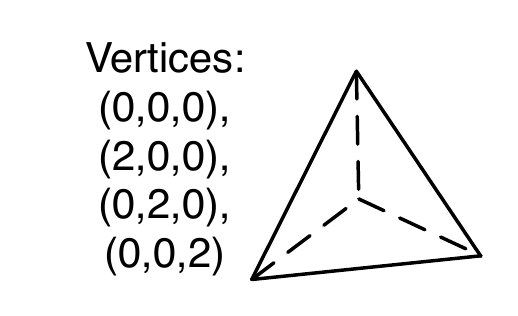} & \includegraphics[scale=0.7]{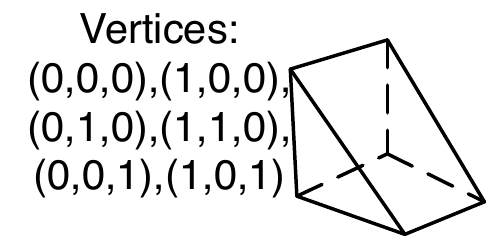}\\ \hline  
$\pn 3$ embedded in $\pn 3$ & $\pn 3$  embedded in $\pn 9$  & A $\pn 2$-bundle over $\pn 1$ embedded in $\pn 5$\\ \hline

\includegraphics[scale=0.7]{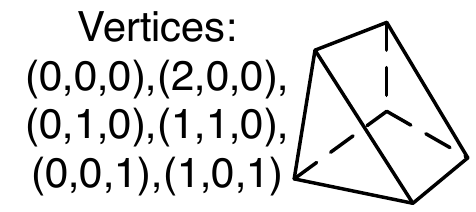}&\includegraphics[scale=0.7]{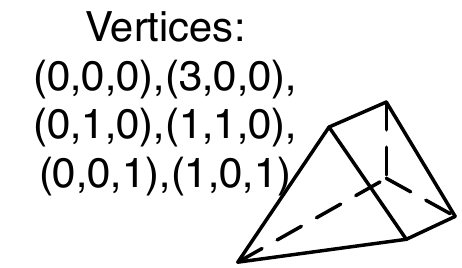}&\includegraphics[scale=0.7]{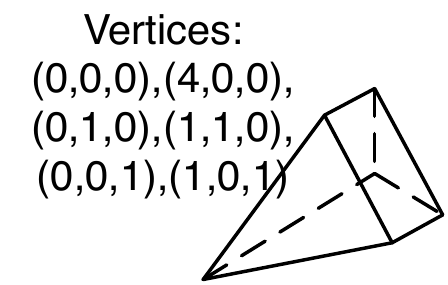}\\ \hline
A $\pn 2$-bundle over $\pn 1$ embedded in $\pn 6$ & A $\pn 2$-bundle over $\pn 1$ embedded in $\pn 7$&A $\pn 2$-bundle over $\pn 1$ embedded in $\pn 8$ \\ \hline

\includegraphics[scale=0.7]{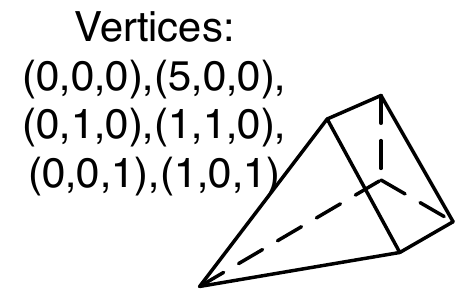} &\includegraphics[scale=0.7]{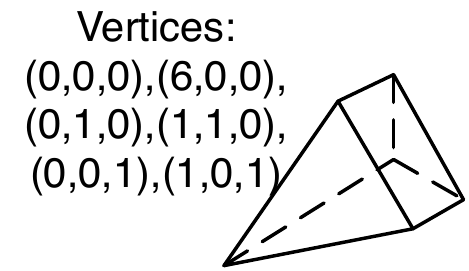}&\includegraphics[scale=0.7]{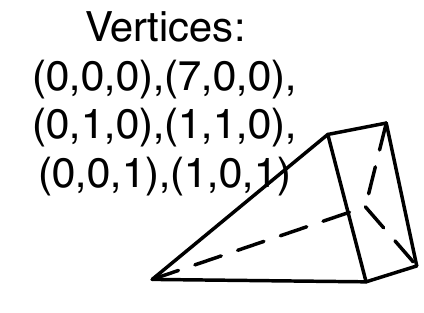} \\ \hline
A $\pn 2$-bundle over $\pn 1$ embedded in $\pn 9$ & A $\pn 2$-bundle over $\pn 1$ embedded in $\pn {10}$&A $\pn 2$-bundle over $\pn 1$ embedded in $\pn {11}$\\ \hline

\includegraphics[scale=0.7]{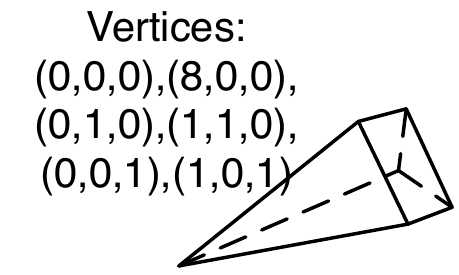} & \includegraphics[scale=0.7]{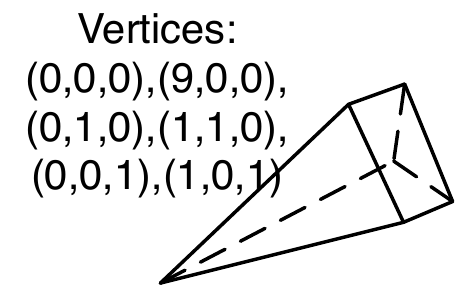} &\includegraphics[scale=0.7]{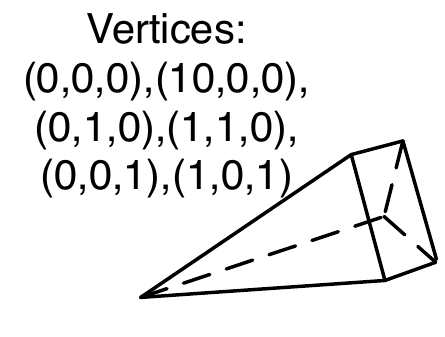} \\ \hline
A $\pn 2$-bundle over $\pn 1$ embedded in $\pn {12}$ & A $\pn 2$-bundle over $\pn 1$ embedded in $\pn {13}$&A $\pn 2$-bundle over $\pn 1$ embedded in $\pn {14}$ \\ \hline

\includegraphics[scale=0.7]{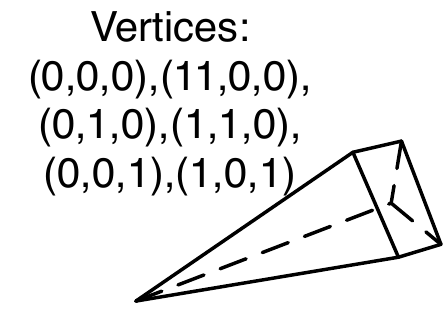} & \includegraphics[scale=0.7]{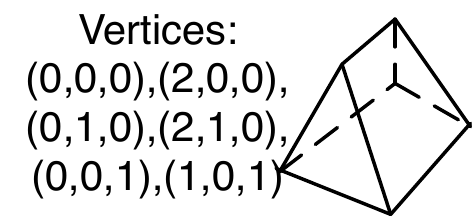}&\includegraphics[scale=0.7]{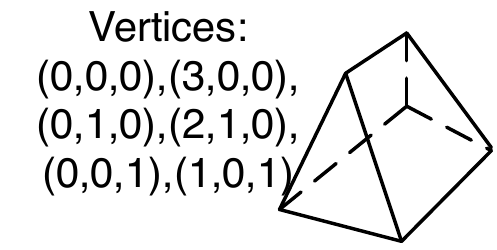} \\ \hline
A $\pn 2$-bundle over $\pn 1$ embedded in $\pn {15}$ & A $\pn 2$-bundle over $\pn 1$ embedded in $\pn 7$ &A $\pn 2$-bundle over $\pn 1$ embedded in $\pn 8$ \\ \hline
\end{tabular}

\begin{tabular}{|p{3.4cm}|p{3.4cm}|p{3.4cm}|}
\hline
\includegraphics[scale=0.7]{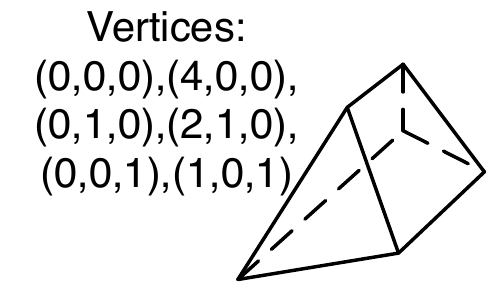} &\includegraphics[scale=0.7]{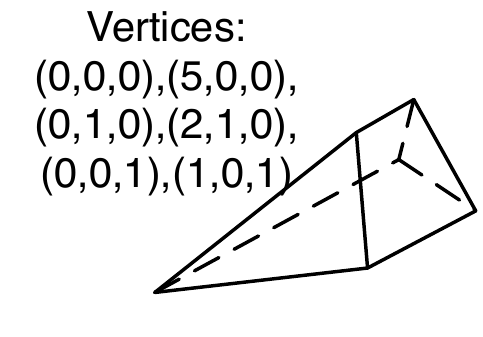}&\includegraphics[scale=0.7]{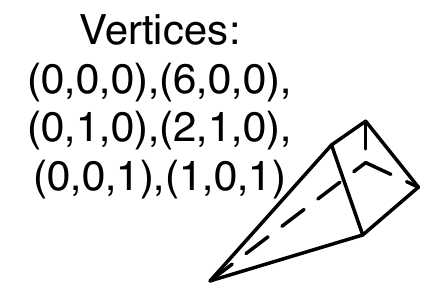}\\ \hline
A $\pn 2$-bundle over $\pn 1$ embedded in $\pn 9$ & A $\pn 2$-bundle over $\pn 1$ embedded in $\pn {10}$&A $\pn 2$-bundle over $\pn 1$ embedded in $\pn {11}$ \\ \hline

\includegraphics[scale=0.7]{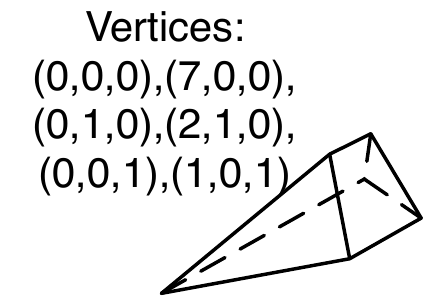}& \includegraphics[scale=0.7]{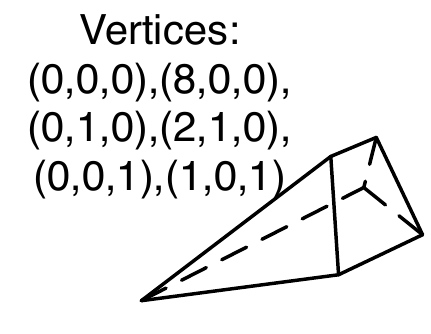}&\includegraphics[scale=0.7]{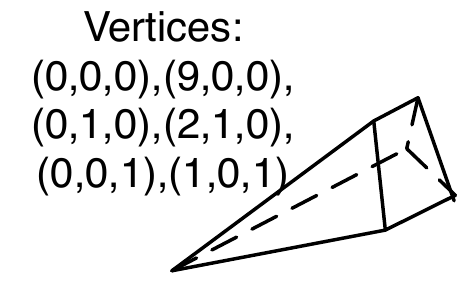}\\ \hline
A $\pn 2$-bundle over $\pn 1$ embedded in $\pn {12}$ & A $\pn 2$-bundle over $\pn 1$ embedded in $\pn {13}$ &A $\pn 2$-bundle over $\pn 1$ embedded in $\pn {14}$\\ \hline

\includegraphics[scale=0.7]{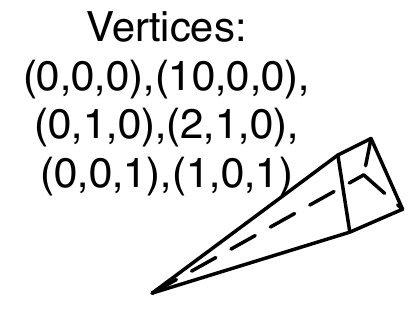} & \includegraphics[scale=0.7]{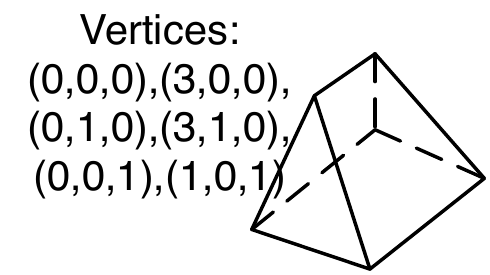} & \includegraphics[scale=0.7]{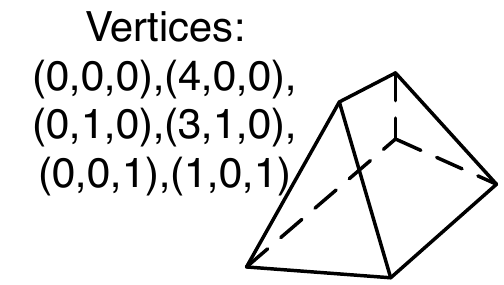}\\ \hline
A $\pn 2$-bundle over $\pn 1$ embedded in $\pn {15}$ & A $\pn 2$-bundle over $\pn 1$ embedded in $\pn 9$ & A $\pn 2$-bundle over $\pn 1$ embedded in $\pn {10}$ \\ \hline
 
\includegraphics[scale=0.7]{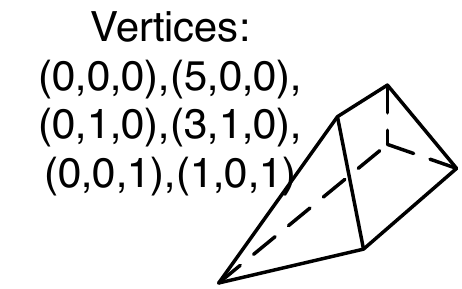} & \includegraphics[scale=0.7]{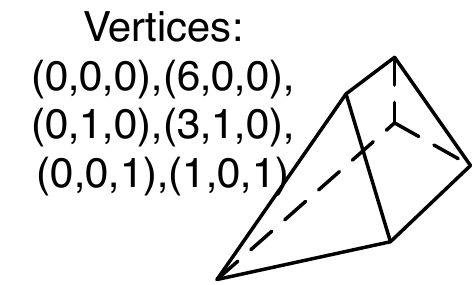} & \includegraphics[scale=0.7]{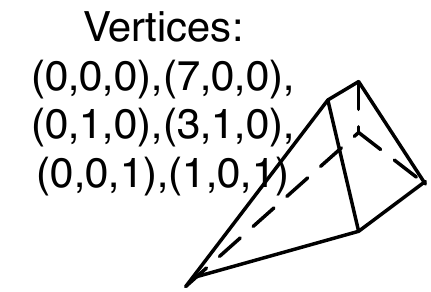}\\ \hline
A $\pn 2$-bundle over $\pn 1$ embedded in $\pn {11}$ & A $\pn 2$-bundle over $\pn 1$ embedded in $\pn {12}$& A $\pn 2$-bundle over $\pn 1$ embedded in $\pn {13}$ \\ \hline 

\includegraphics[scale=0.7]{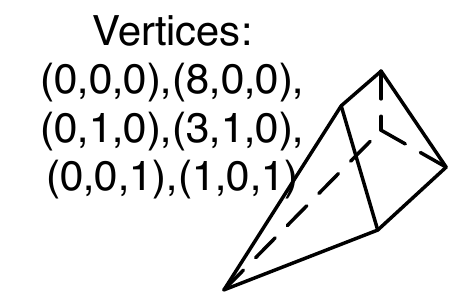} & \includegraphics[scale=0.7]{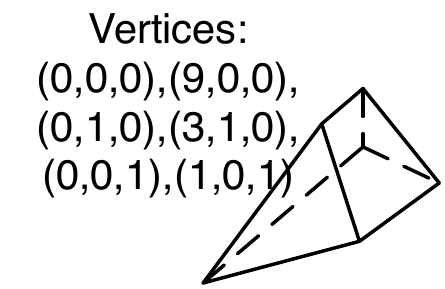} & \includegraphics[scale=0.7]{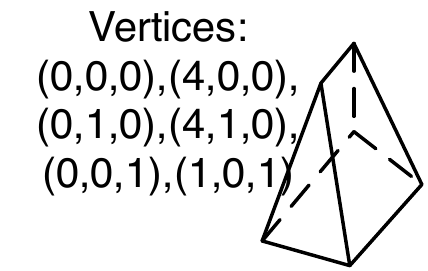} \\ \hline
A $\pn 2$-bundle over $\pn 1$ embedded in $\pn {14}$ & A $\pn 2$-bundle over $\pn 1$ embedded in $\pn {15}$&A $\pn 2$-bundle over $\pn 1$ embedded in $\pn {11}$\\ \hline 

\includegraphics[scale=0.7]{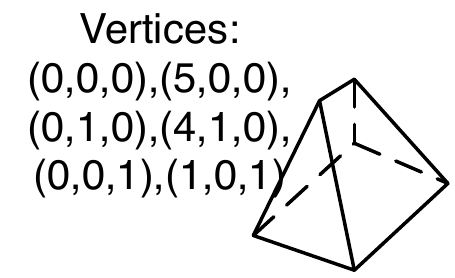} & \includegraphics[scale=0.7]{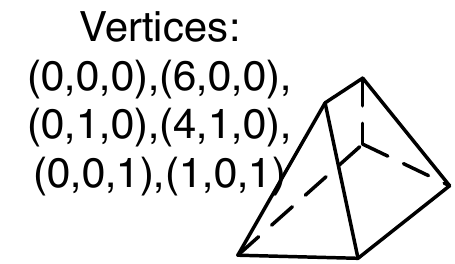} &\includegraphics[scale=0.7]{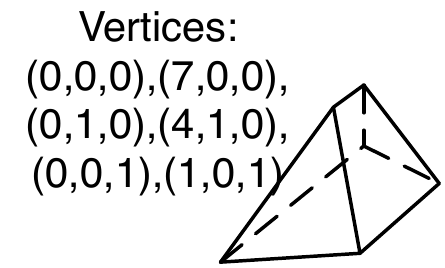} \\ \hline
A $\pn 2$-bundle over $\pn 1$ embedded in $\pn {12}$ & A $\pn 2$-bundle over $\pn 1$ embedded in $\pn {13}$ & A $\pn 2$-bundle over $\pn 1$ embedded in $\pn {14}$ \\ \hline 
\end{tabular}

\newpage
\begin{tabular}{|p{3.4cm}|p{3.4cm}|p{3.4cm}|}
\hline

\includegraphics[scale=0.7]{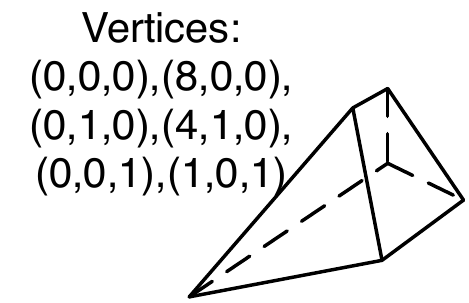} & \includegraphics[scale=0.7]{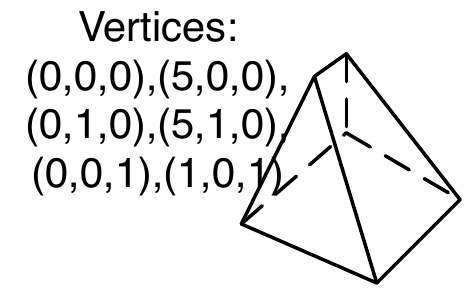} &\includegraphics[scale=0.7]{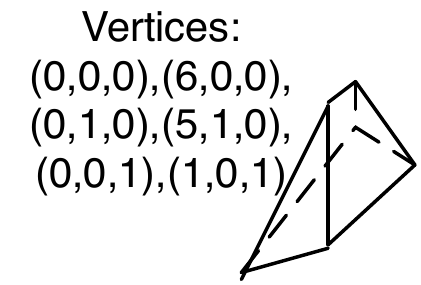} \\ \hline
A $\pn 2$-bundle over $\pn 1$ embedded in $\pn {15}$ & A $\pn 2$-bundle over $\pn 1$ embedded in $\pn {13}$& A $\pn 2$-bundle over $\pn 1$ embedded in $\pn {14}$ \\ \hline

\includegraphics[scale=0.7]{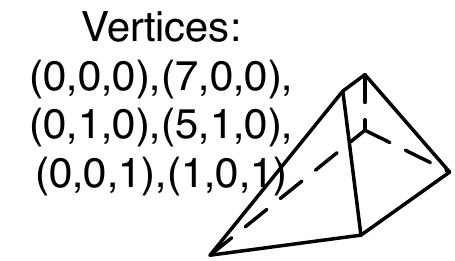} &\includegraphics[scale=0.7]{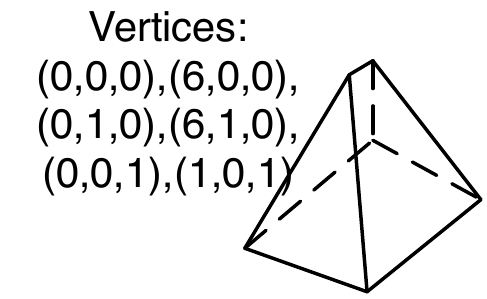} & \includegraphics[scale=0.7]{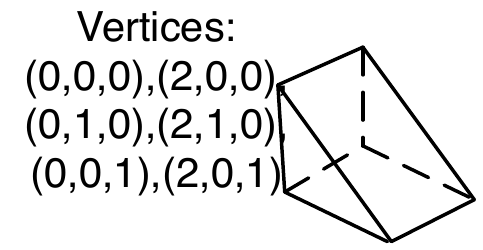}\\ \hline
A $\pn 2$-bundle over $\pn 1$ embedded in $\pn {15} $ & A $\pn 2$-bundle over $\pn 1$ embedded in $\pn {15}$ & A $\pn 2$-bundle over $\pn 1$ embedded in $\pn 8$ \\ \hline

\includegraphics[scale=0.7]{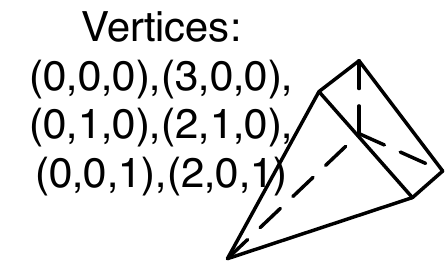} & \includegraphics[scale=0.7]{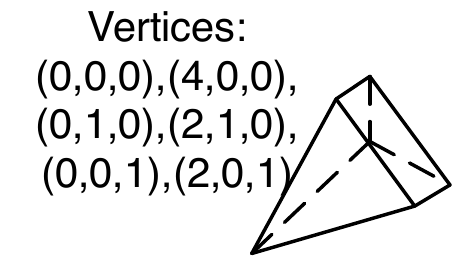} & \includegraphics[scale=0.7]{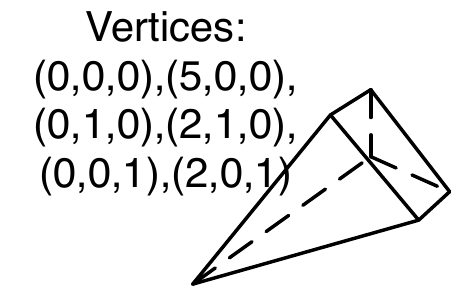}\\ \hline
A $\pn 2$-bundle over $\pn 1$ embedded in $\pn 9$ & A $\pn 2$-bundle over $\pn 1$ embedded in $\pn {10}$ & A $\pn 2$-bundle over $\pn 1$ embedded in $\pn {11}$\\ \hline 

\includegraphics[scale=0.7]{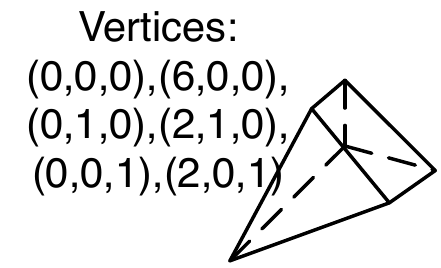} & \includegraphics[scale=0.7]{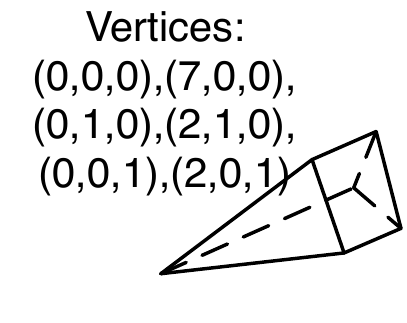} & \includegraphics[scale=0.7]{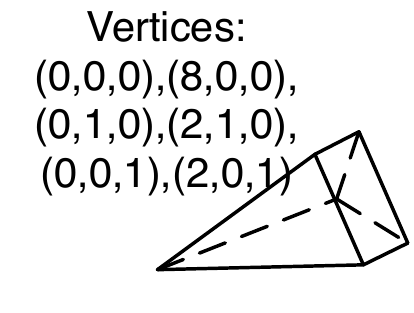} \\ \hline
A $\pn 2$-bundle over $\pn 1$ embedded in $\pn {12}$ & A $\pn 2$-bundle over $\pn 1$ embedded in $\pn {13}$ & A $\pn 2$-bundle over $\pn 1$ embedded in $\pn {14}$ \\ \hline

\includegraphics[scale=0.7]{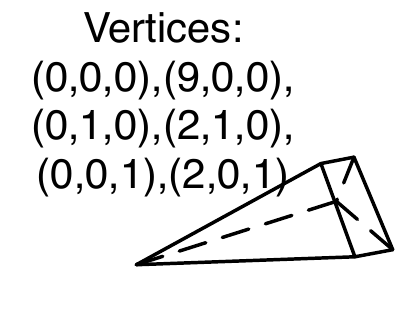} & \includegraphics[scale=0.7]{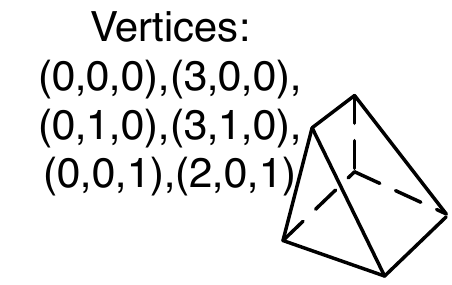}&\includegraphics[scale=0.7]{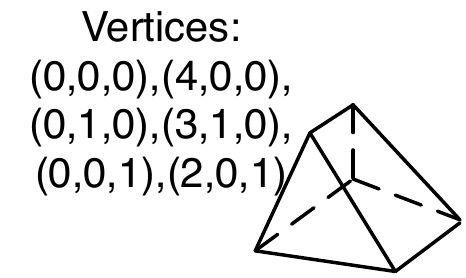} \\ \hline
A $\pn 2$-bundle over $\pn 1$ embedded in $\pn {15}$ & A $\pn 2$-bundle over $\pn 1$ embedded in $\pn {10}$ & A $\pn 2$-bundle over $\pn 1$ embedded in $\pn {11}$ \\ \hline 

\includegraphics[scale=0.7]{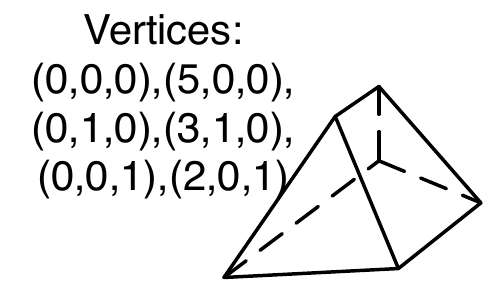} &\includegraphics[scale=0.7]{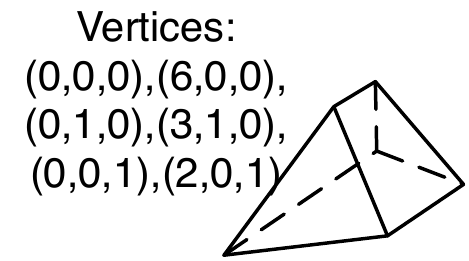} & \includegraphics[scale=0.7]{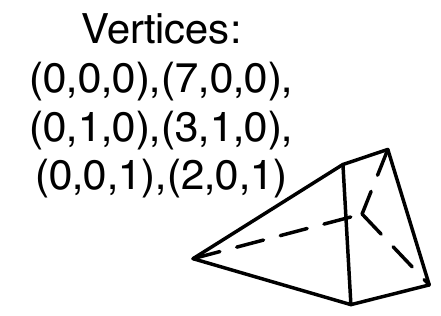}\\ \hline
A $\pn 2$-bundle over $\pn 1$ embedded in $\pn {12}$ & A $\pn 2$-bundle over $\pn 1$ embedded in $\pn {13}$ & A $\pn 2$-bundle over $\pn 1$ embedded in $\pn {14}$ \\ \hline
\end{tabular}

\newpage
\begin{tabular}{|p{3.4cm}|p{3.4cm}|p{3.4cm}|}
\hline

\includegraphics[scale=0.7]{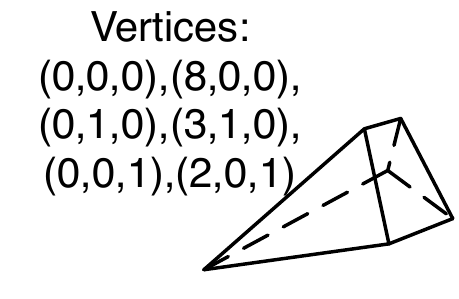} &\includegraphics[scale=0.7]{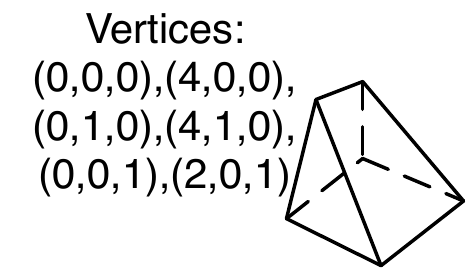}& \includegraphics[scale=0.7]{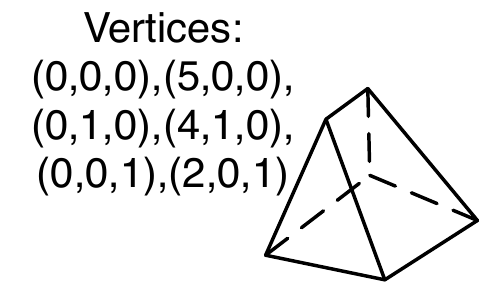} \\ \hline
A $\pn 2$-bundle over $\pn 1$ embedded in $\pn {15}$ & A $\pn 2$-bundle over $\pn 1$ embedded in $\pn {12}$& A $\pn 2$-bundle over $\pn 1$ embedded in $\pn {13}$\\ \hline

\includegraphics[scale=0.7]{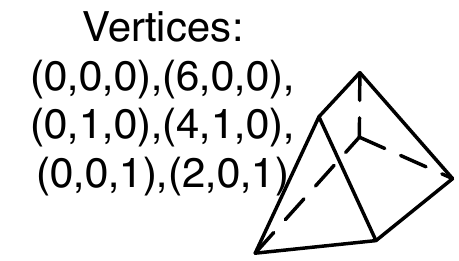} & \includegraphics[scale=0.7]{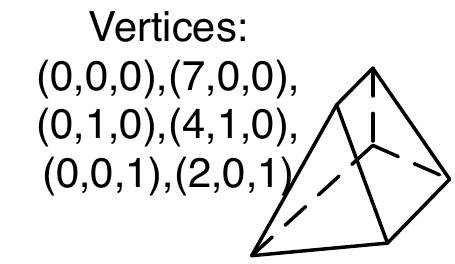} & \includegraphics[scale=0.7]{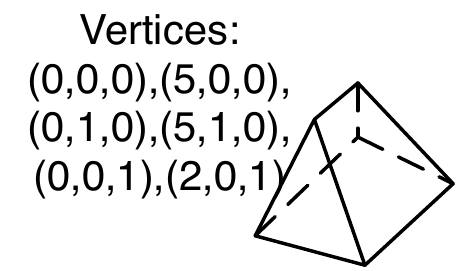}\\ \hline
 A $\pn 2$-bundle over $\pn 1$ embedded in $\pn {14}$ & A $\pn 2$-bundle over $\pn 1$ embedded in $\pn {15}$ & A $\pn 2$-bundle over $\pn 1$ embedded in $\pn {14}$\\ \hline 

\includegraphics[scale=0.7]{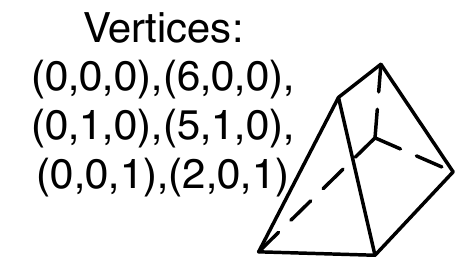} & \includegraphics[scale=0.7]{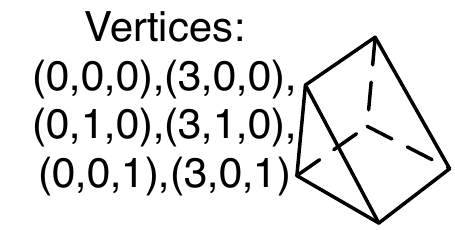} & \includegraphics[scale=0.7]{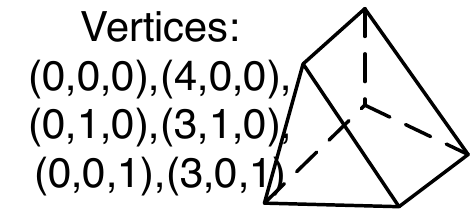}    \\ \hline
A $\pn 2$-bundle over $\pn 1$ embedded in $\pn {15}$& A $\pn 2$-bundle over $\pn 1$ embedded in $\pn {11}$ & A $\pn 2$-bundle over $\pn 1$ embedded in $\pn {12}$ \\ \hline

\includegraphics[scale=0.7]{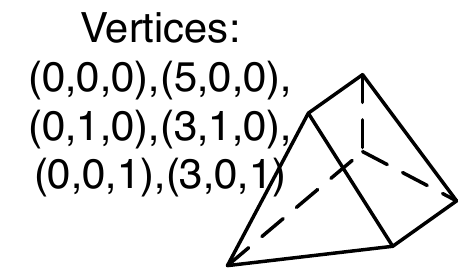} & \includegraphics[scale=0.7]{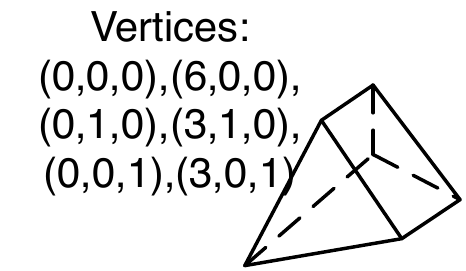} & \includegraphics[scale=0.7]{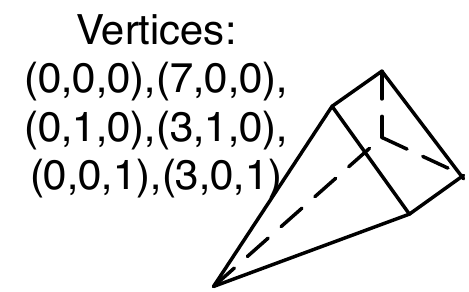}\\ \hline
A $\pn 2$-bundle over $\pn 1$ embedded in $\pn {13}$ &A $\pn 2$-bundle over $\pn 1$ embedded in $\pn {14}$& A $\pn 2$-bundle over $\pn 1$ embedded in $\pn {15}$ \\ \hline 

\includegraphics[scale=0.7]{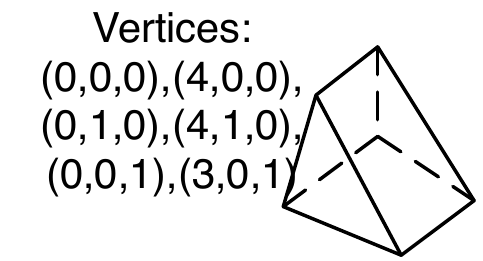} &\includegraphics[scale=0.7]{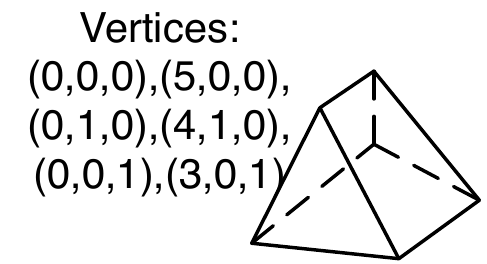} & \includegraphics[scale=0.7]{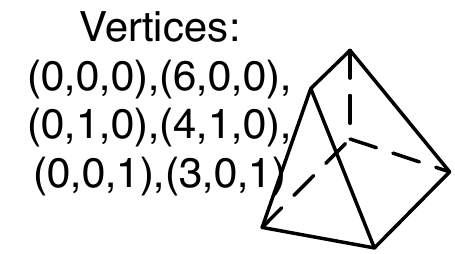}\\ \hline 
A $\pn 2$-bundle over $\pn 1$ embedded in $\pn {13}$ & A $\pn 2$-bundle over $\pn 1$ embedded in $\pn {14}$& A $\pn 2$-bundle over $\pn 1$ embedded in $\pn {15}$\\ \hline

\includegraphics[scale=0.7]{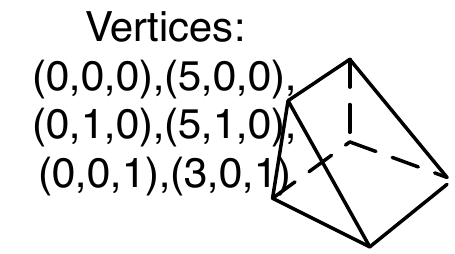} & \includegraphics[scale=0.7]{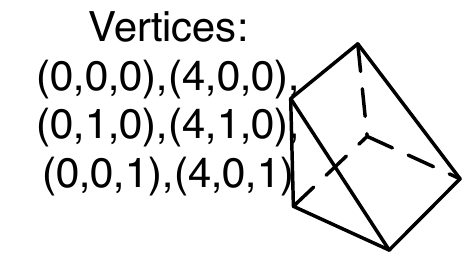} & \includegraphics[scale=0.7]{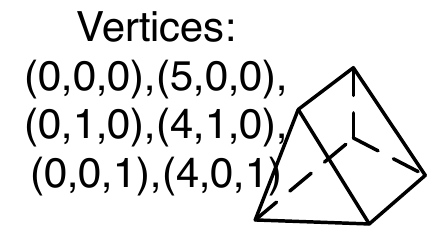}\\ \hline
A $\pn 2$-bundle over $\pn 1$ embedded in $\pn {15}$& A $\pn 2$-bundle over $\pn 1$ embedded in $\pn {14}$ & A $\pn 2$-bundle over $\pn 1$ embedded in $\pn {15}$\\ \hline 
\includegraphics[scale=0.7]{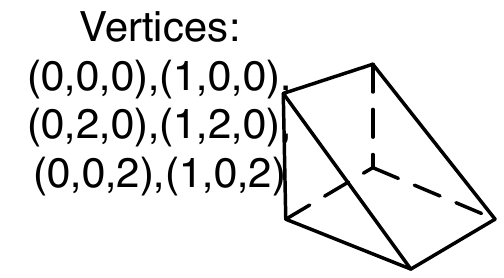}&\includegraphics[scale=0.7]{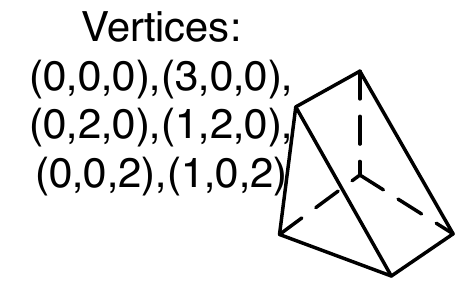}& \includegraphics[scale=0.7]{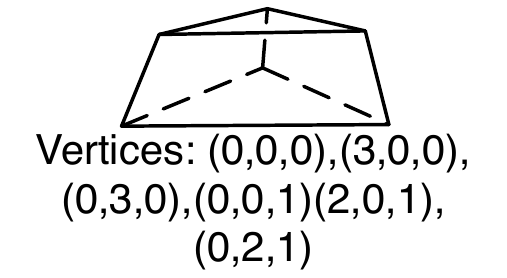} \\ \hline
A $\pn 2$-bundle over $\pn 1$ embedded in $\pn {11}$ & A $\pn 2$-bundle over $\pn 1$ embedded in $\pn {15}$ & A $\pn 1$-bundle over $\pn 2$ embedded in $\pn {15}$\\ \hline
\end{tabular}
\newpage
\begin{tabular}{|p{3.4cm}|p{3.4cm}|p{3.4cm}|}
\hline

\includegraphics[scale=0.7]{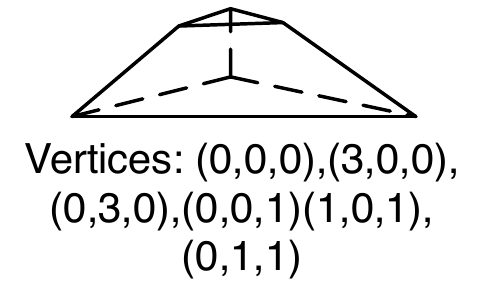} & \includegraphics[scale=0.7]{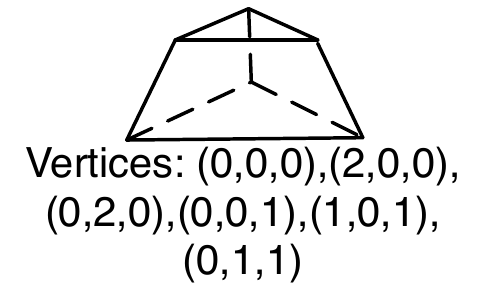} & \includegraphics[scale=0.7]{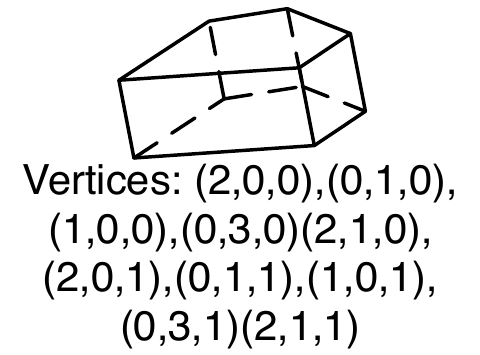}\\ \hline
A $\pn 1$-bundle over $\pn 2$ embedded in $\pn {12}$ & A $\pn 1$-bundle over $\pn 2$  embedded in $\pn 8$& A $\pn 1$-bundle over $\Bl_2(\pn 2)$  embedded in $\pn {15}$ \\ \hline

\includegraphics[scale=0.7]{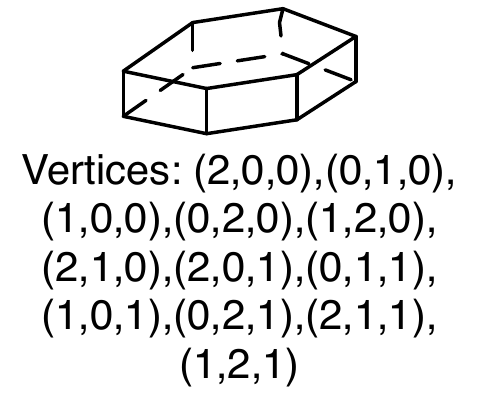} &\includegraphics[scale=0.7]{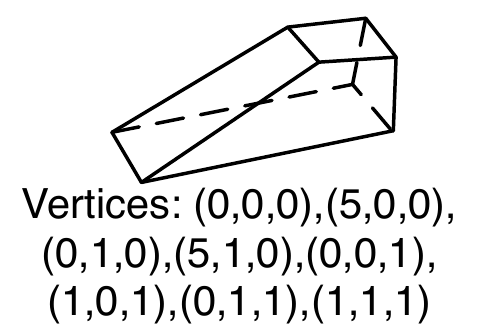} & \includegraphics[scale=0.7]{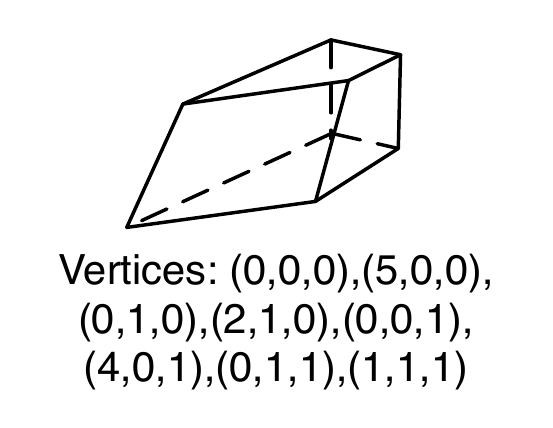} \\ \hline
A $\pn 1$-bundle over $\Bl_3(\pn 2)$  embedded in $\pn {13}$ & A $\pn 1$-bundle over $\F_4$  embedded in $\pn {15}$ & A $\pn 1$-bundle over $\F_3$  embedded in $\pn {15}$\\ \hline

\includegraphics[scale=0.7]{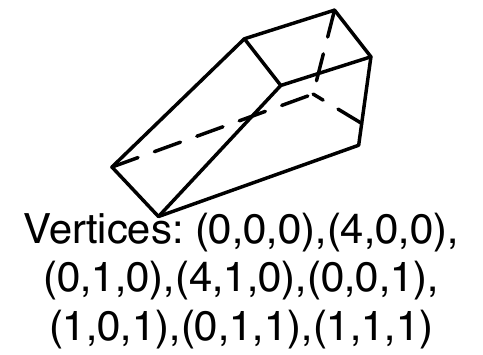} & \includegraphics[scale=0.7]{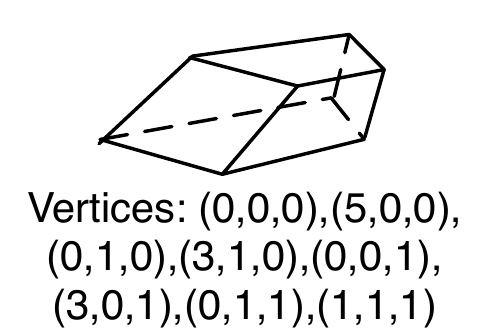} & \includegraphics[scale=0.7]{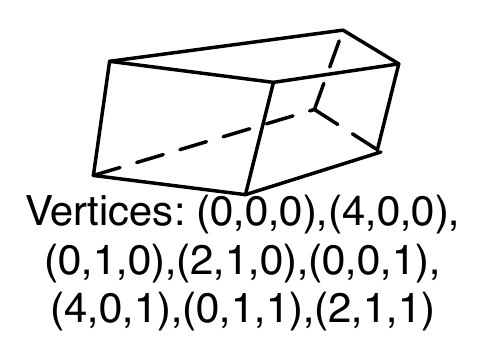} \\ \hline
A $\pn 1$-bundle over $\F_3$  embedded in $\pn {13}$ &A $\pn 1$-bundle over $\F_2$  embedded in $\pn {15}$ & A $\pn 1$-bundle over $\F_2$ embedded in $\pn {15}$  \\ \hline

\includegraphics[scale=0.7]{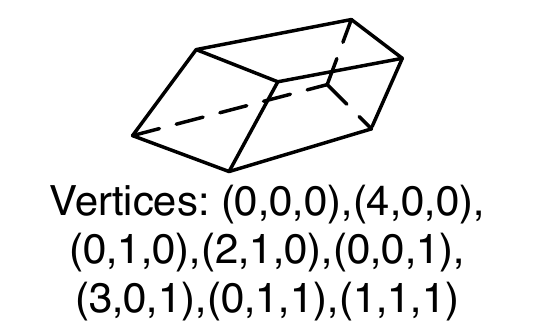} & \includegraphics[scale=0.7]{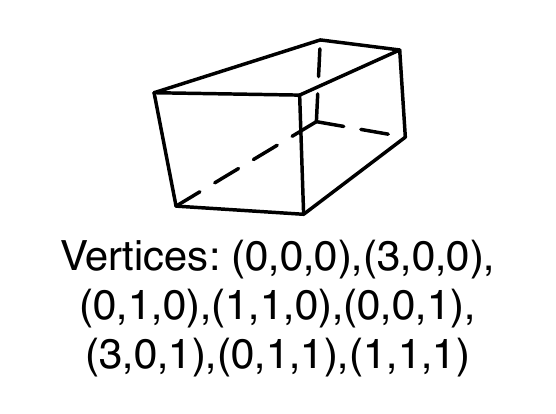} & \includegraphics[scale=0.7]{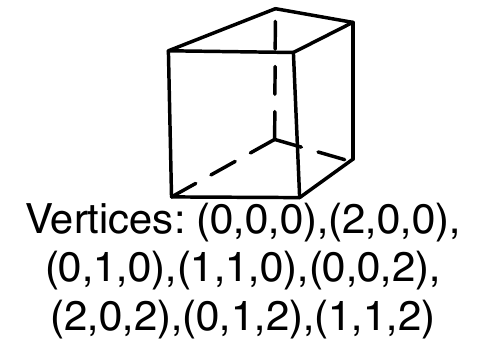}  \\ \hline
A $\pn 1$-bundle over $\F_2$ embedded in $\pn {13}$  & A $\pn 1$-bundle over $\F_2$ embedded in $\pn {11}$& A $\pn 1$-bundle over $\F_1$ embedded in $\pn {14}$\\ \hline

\includegraphics[scale=0.7]{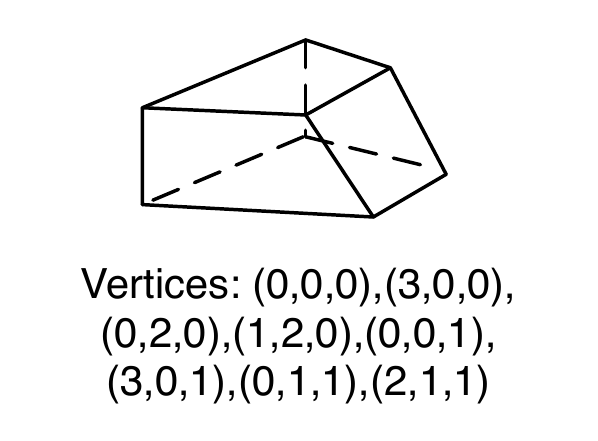}& \includegraphics[scale=0.7]{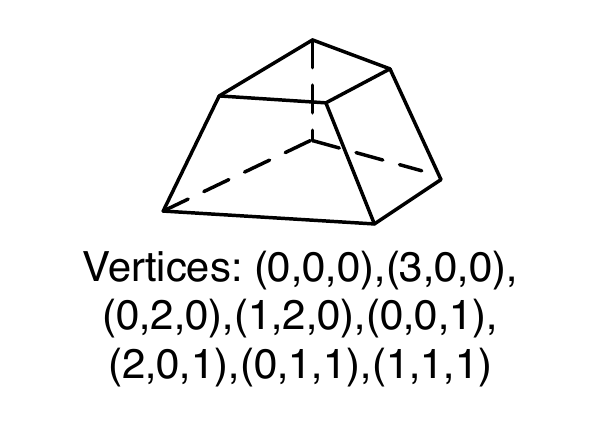} & \includegraphics[scale=0.7]{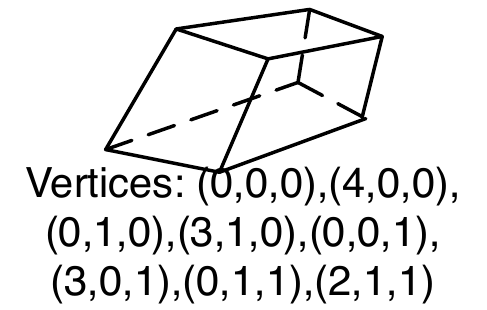}  \\ \hline 
A $\pn 1$-bundle over $\F_1$ embedded in $\pn {15}$ & A $\pn 1$-bundle over $\F_1$ embedded in $\pn {13}$ & A $\pn 1$-bundle over $\F_1$ embedded in $\pn {15}$ \\ \hline

\end{tabular}
\newpage
\begin{tabular}{|p{3.4cm}|p{3.4cm}|p{3.4cm}|}
\hline
\includegraphics[scale=0.7]{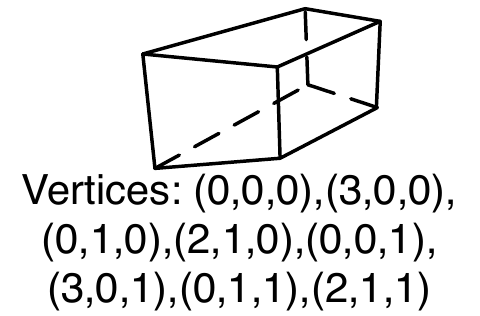}&\includegraphics[scale=0.7]{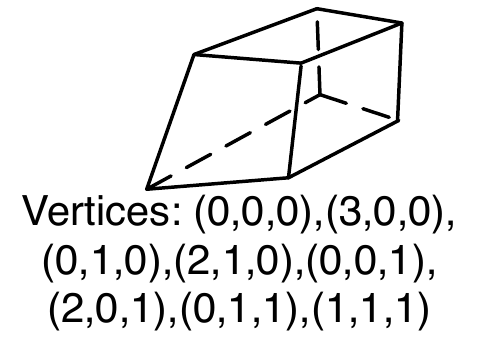} & \includegraphics[scale=0.7]{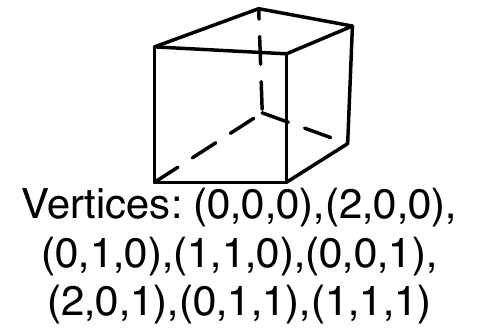}  \\ \hline
A $\pn 1$-bundle over $\F_1$ embedded in $\pn {13}$ &A $\pn 1$-bundle over $\F_1$ embedded in $\pn {11}$  & A $\pn 1$-bundle over $\F_1$ embedded in $\pn {9}$\\ \hline 

\includegraphics[scale=0.7]{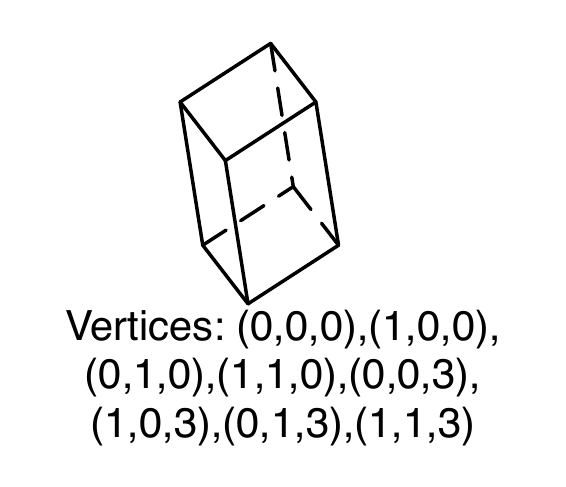} & \includegraphics[scale=0.7]{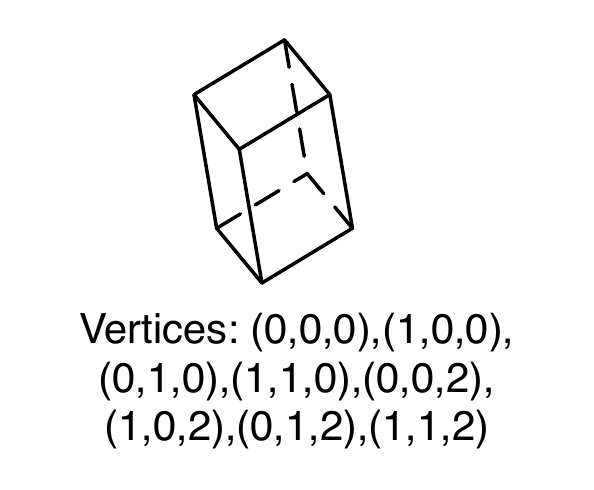} & \includegraphics[scale=0.7]{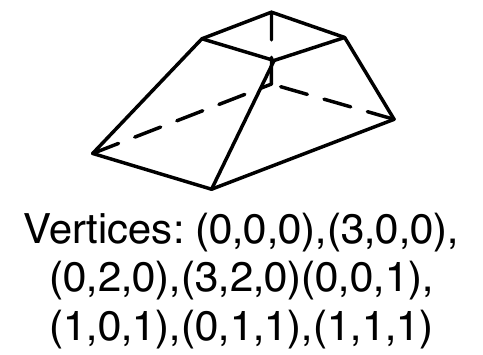}\\ \hline
A $\pn 1$-bundle over $\F_0$ embedded in $\pn {15}$ & A $\pn 1$-bundle over $\F_0$ embedded in $\pn {11}$ & A $\pn 1$-bundle over $\F_0$ embedded in $\pn {15}$\\ \hline

\includegraphics[scale=0.7]{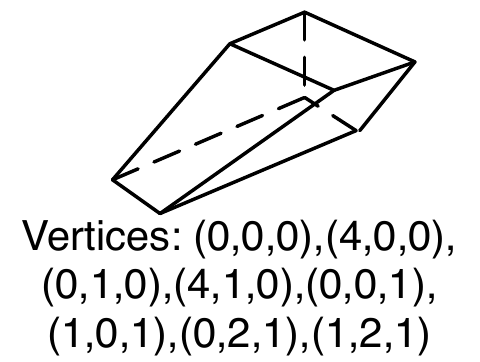} & \includegraphics[scale=0.7]{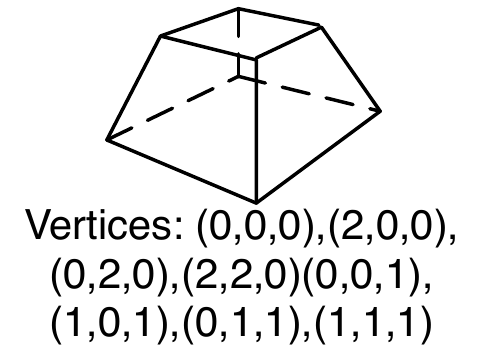}  & \includegraphics[scale=0.7]{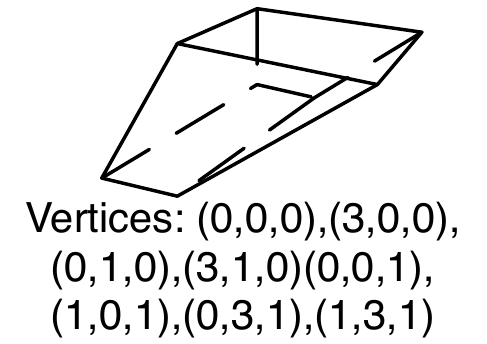}\\ \hline
A $\pn 1$-bundle over $\F_0$  embedded in $\pn {15}$ & A $\pn 1$-bundle over $\F_0$  embedded in $\pn {12}$ & A $\pn 1$-bundle over $\F_0$ embedded in $\pn {15}$ \\ \hline 

\includegraphics[scale=0.7]{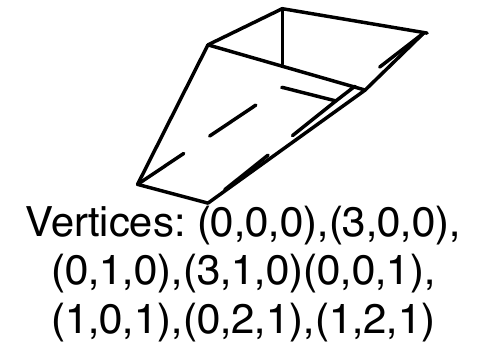}&  \includegraphics[scale=0.7]{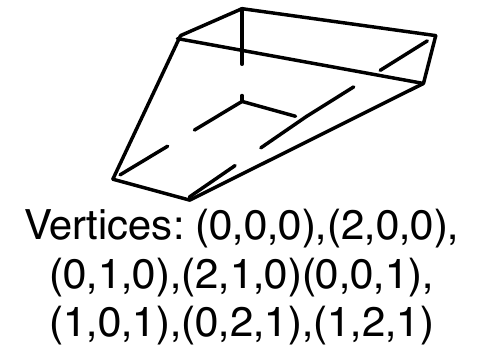}& \includegraphics[scale=0.7]{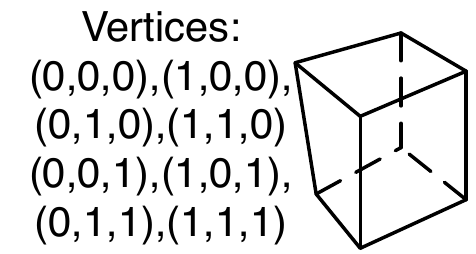}\\ \hline
A $\pn 1$-bundle over $\F_0$  embedded in $\pn {13}$ & A $\pn 1$-bundle over $\F_0$  embedded in $\pn {11}$ & A $\pn 1$-bundle over $\F_0$  embedded in $\pn {7}$\\ \hline
\includegraphics[scale=0.7]{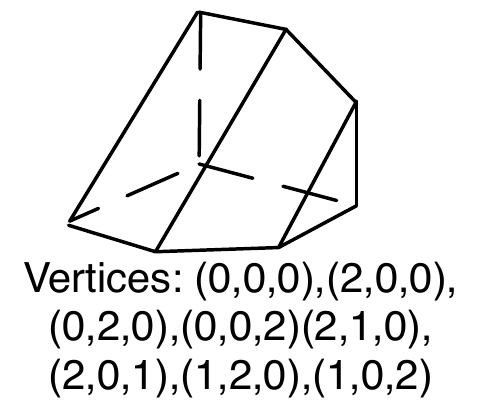} & \includegraphics[scale=0.7]{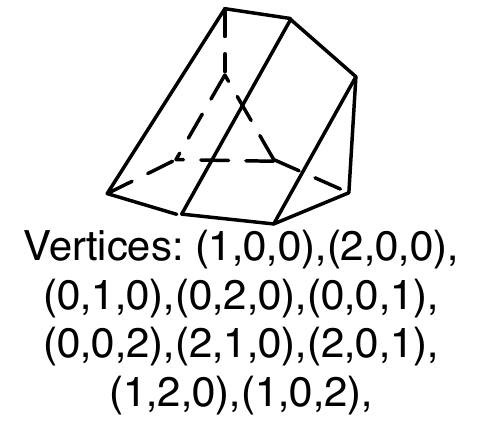} & \includegraphics[scale=0.7]{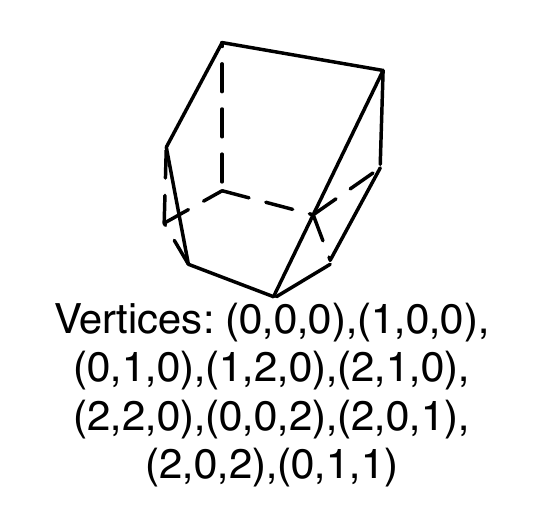}\\ \hline
The blow-up at one point of a $\pn 2$-bundle over $\pn 1$ embedded in $\pn {14}$ & The blow-up at two points of a $\pn 2$-bundle over $\pn 1$ embedded in $\pn {13}$& The blow-up at two points of a $\pn 1$-bundle over $\pn 2$ embedded in $\pn {15}$\\ \hline
\end{tabular}

\newpage
\begin{tabular}{|p{3.4cm}|p{3.4cm}|p{3.4cm}|}
\hline
&\includegraphics[scale=0.7]{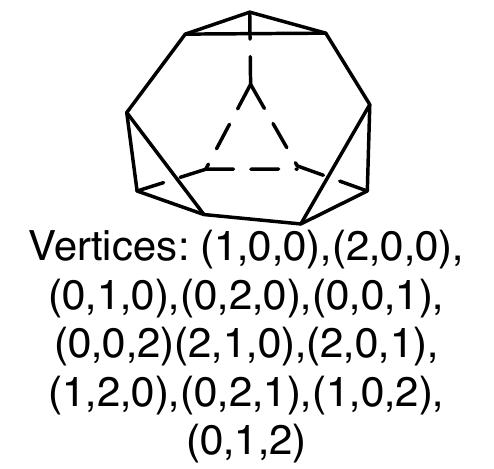}&  \\ \hline
&The blow-up at four points of $\pn 3$ embedded in $\pn {15}$& \\ \hline
\end{tabular}
\end{center}





\section{acknowledgements}
First and foremost I would like to thank my advisor Professor Sandra di Rocco for her fantastic support and for introducing me to toric geometry, convex combinatorics and the intriguing connection between the two fields. Moreover I would like to thank the Department of Mathematics at KTH in Stockholm and the Swedish Vetenskaps r\aa det for giving me the opportunity and means to nourish this newfound interest. Finally I would like to thank the anonymous referees for their comments which greatly improved this paper.

\bibliographystyle{plain}      
\bibliography{Refsart}   

\end{document}